\documentclass{siamltex}
\usepackage[utf8]{inputenc}
\usepackage{xspace}
\usepackage{comment}
\usepackage[normalem]{ulem}
\usepackage{amstext,amsmath,amsfonts,amssymb,amsbsy}
\usepackage{mathrsfs}
\usepackage{yfonts}
\usepackage[shortlabels]{enumitem}
\usepackage{graphicx}
\usepackage{color} 
\usepackage{hyperref}
\usepackage{algpseudocode}
\usepackage{algorithm}
\newtheorem{property}[theorem]{Property}
\newtheorem{remark}[theorem]{Remark}
\newtheorem{example}[theorem]{Example}
\newtheorem{claim}[theorem]{Claim}

\newcounter{casecount}
\newenvironment{case}{\refstepcounter{casecount}\textbf{Case \arabic{casecount}:}}{}

\usepackage[shadow,colorinlistoftodos,prependcaption,spanish,textsize=footnotesize]{todonotes}
\usepackage{xargs}   % Use more than one optional parameter in a new commands
\usepackage{xcolor} % Colors

\newcommandx{\marcetodo}[2][1=]{} %{\todo[linecolor=lime!70!black,backgroundcolor=lime!70!black!25,bordercolor=lime!70!black,#1]{#2}}
\newcommandx{\pedrotodo}[2][1=]{} %{\todo[linecolor=magenta,backgroundcolor=magenta!25,bordercolor=magenta,#1]{#2}}
\newcommandx{\sebatodo}[2][1=]{} %{\todo[linecolor=blue,backgroundcolor=blue!25,bordercolor=blue,#1]{#2}}
\usepackage{etoolbox}
\AtBeginEnvironment{proof}{\setcounter{casecount}{0}}

\usepackage{bm}
\renewcommand{\vec}[1]{\bm{#1}}               % vector notation
\newcommand{\A}{\mathbb{A}}
\def\R{\ensuremath{\mathbb{R}}\xspace}
\def\Z{\ensuremath{\mathbb{Z}}\xspace}
\def\l{\ensuremath{\ell}\xspace}
\def\V{\ensuremath{V}\xspace}  %% the Bspline space
\def\calB{\ensuremath{\mathcal{B}}\xspace}  %% the Bspline space
\def\calL{\ensuremath{\mathcal{L}}\xspace}  %% the Bspline space
\def\calN{\ensuremath{\mathcal{N}}}  %% the ``reach'' of a function
\def\calH{\ensuremath{\mathcal{H}}\xspace}
\def\calR{\ensuremath{\mathcal{R}}\xspace}
\def\calC{\ensuremath{\mathcal{C}}\xspace}
\def\calM{\ensuremath{\mathcal{M}}\xspace}

\def\calI{\ensuremath{\mathcal{I}}\xspace}
\def\calJ{\ensuremath{\mathcal{J}}\xspace}
\def\calO{\ensuremath{\mathcal{O}}\xspace}
\def\calF{\ensuremath{\mathcal{F}}\xspace}
\def\calA{\ensuremath{\mathcal{A}}\xspace}
\def\calP{\ensuremath{\mathcal{P}}\xspace}
\def\bbI{\ensuremath{\mathbb{I}}\xspace}
\def\bbB{\ensuremath{\mathbb{B}}\xspace}

\def\bbI{\ensuremath{\mathbb{I}}\xspace}

\def\bbV{\ensuremath{\mathbb{V}}\xspace}
\def\bbW{\ensuremath{\mathbb{W}}\xspace}
\newcommand{\ovch}[4][\hspace{0pt}]{\calO^{#1}(#2,#3,#4)} %

\newcommand{\ball}[3]{B(#1,#2,#3)}

\newcommand{\vphi}{\varphi}

\newcommand{\Cle}{\lesssim}

\DeclareMathOperator{\dim1}{dim}

\DeclareMathOperator{\diam}{diam}
\DeclareMathOperator{\depth}{depth}
\DeclareMathOperator{\dist}{dist}
\DeclareMathOperator{\spn}{span}   
\DeclareMathOperator{\gap}{gap}
\DeclareMathOperator{\sat}{sat}

\newcommand{\ch}[2][\hspace{0pt}]{\operatorname{ch}^{#1}#2} % children
\newcommand{\fdiv}[3][\hspace{0pt}]{\operatorname{D}^{#1}_{#2}({#3})}
\newcommand{\cdiv}[3][\hspace{0pt}]{\operatorname{\bar{D}}^{#1}_{#2}({#3})}
\newcommand{\ipro}[3][\hspace{0pt}]{\operatorname{M}^{#1}_{#2}({#3})}

\newcommand{\fgap}[2][\calH]{\operatorname{gap}_{#1}#2}
\DeclareMathOperator{\prt}{par}

\DeclareMathOperator{\anc}{anc}
\DeclareMathOperator{\descs}{dsc}

\def\dsc{\ll} % is descendant of
\def\ab{\subset} % is absorbed by
\def\refi{\succ} % is a refinement of

\newcommand{\mli}[1]{\hat{#1}} %linear independization of 
\def\dbigcup{\mathop{\dot{\bigcup}}} 
\def\dcup{\mathbin{\dot{\cup}}}

\newcommand{\pedro}[1]{#1} %{{\color{magenta}#1}}
\newcommand{\marce}[1]{#1} %{{\color{lime!70!black}#1}}
\newcommand{\seba}[1]{#1} %{{\color{blue}#1}}

\title{A new perspective on hierarchical spline spaces for adaptivity%
\thanks{This work was Partially supported by Agencia Nacional de Promoci\'on Científica
y Tecnol\'ogica, through grants PICT-2014-2522, PICT-2016-1983, by CONICET through PIP 2015 11220150100661, and by Universidad Nacional del Litoral through grants 
CAI+D 2016-50420150100022LI % Pedro
and 2016-50020150100074LI. % Marce
}%
}
\author{
Marcelo Actis$^\dagger$%
\and
Pedro Morin$^\dagger$%
\and
M.~Sebastian Pauletti%
\thanks{Facultad de Ingenieria Quimica, Universidad Nacional del Litoral, 
Santiago del Estero 2829, S3000AOM Santa Fe, Argentina. Researcher of
  Consejo Nacional de Investigaciones cient\'{i}ficas y t\'ecnicas (CONICET).
}
}

\begin{document}

\maketitle
\begin{center}
\today
\end{center}

\begin{abstract}
We introduce a framework for spline spaces of hierarchical type, based on a parent-children relation, which is very convenient for the analysis as well as the implementation of adaptive isogeometric methods.
Such framework makes it simple to create hierarchical \emph{basis} with \emph{control on the overlapping}.
Linear independence is always desired for the well posedness of the linear systems, and to avoid redundancy.
The control on the overlapping of basis functions from different levels is necessary to close theoretical arguments in the proofs of optimality of adaptive methods.
In order to guarantee linear independence, and to control the overlapping of the basis functions, some basis functions additional to those initially marked must be refined.
However, with our framework and refinement procedures, the complexity of the resulting bases is under control, i.e., the resulting bases have cardinality bounded by the number of initially marked functions.
%%
%More precisely, if we construct a sequence of hierarchical bases ${\cal H}_k$ through subsequent calls to
%$\calH_{k+1} = \Call{GARefine}{\calH_k,\calM_k}$, where $\calM_k \subset \calH_k$ denotes the 
%set of \emph{marked} functions, we obtain
%\[
%\# \calH_R - \# \calH_0\le C \sum_{k=0}^{R-1} \#\cal M_k .
%\]

%%
%The selection is guided by a sequence of sets of marked functions,
%for which we provide a refinement procedure that 
%guarantees a control on the overlapping (gap) of the basis functions as well as a bound on the complexity, i.e. the dimension of the space is uniformly bounded by the 
%cardinality of the history of the marked functions.    
%%
%We dismember the ingredients of the process showing the implications and 
%scopes of each one.
%This sets a language not only to state the appropriate questions but also to 
%start answering them.
\end{abstract}

%%%%%%%%%%%%%%%%%%%%%%%%%%%%%%%%%%%
\section{Introduction}
Adaptive methods are a fundamental computational tool in science and 
engineering to approximate partial differential equations.

For the finite element method (AFEM) there has been a lot a work starting
in the 1980s and 1990s with the design of a posteriori error estimators with
very successful practical result.
In the 2000s adaptive processes have been shown to converge, and to 
exhibit optimal complexity for several stationary PDE.

The adaptive process for stationary PDE can be described with the classical adaptive step 
$$\Call{Solve}{} \to \Call{Estimate}{} \to \Call{Mark}{} \to \Call{Refine}{},$$
where \Call{Solve}{} computes the solution on a discrete space with
basis \calH;
\Call{Estimate}{} computes a posteriori localized error estimators
and \Call{Mark}{} uses the estimators to indicate where more resolution
should be invested in order to obtain maximum benefit.
Let \calM be this indication, then from \calH and \calM the procedure \Call{Refine}{} constructs a new basis $\calH_*$ and thus a new space.
We thus arrive at an adaptive sequence
\begin{equation}\label{e:adaptive_loop}
 \calH_0 \xrightarrow{\calM_0} \calH_1 \xrightarrow{\calM_1} 
\quad\dots\quad 
\xrightarrow{\calM_{R-1}} \calH_R \rightarrow \dots.
\end{equation}
      
A sound theory of adaptivity in the  context of FEM \cite{Nochetto2009,axioms} hinges on the 
adequate design of local estimators and certain 
combinatorial-geometric properties assignable to the underlying mesh that is 
intimately related to the refinement procedure.
The estimator is usually assigned to the mesh elements, the ones with 
the larger estimators are collected in \calM and then refined to obtain a new mesh and thereby a new basis 
with adequate local resolution.

\pedro{Different alternatives have been proposed to obtain adaptive spline methods, such as hierarchical splines, T-splines, LR-splines or PHT-splines. Among them, hierarchical splines such as those in~\cite{VGJS2011,Kraft1997,BG2016} seem to constitute the simplest approaches to obtain adaptive isogeometric methods.}

The structure of B--splines leads naturally to the idea of assigning estimators to and 
refining basis functions rather than elements \cite{KGS2003, BG2018, MNP2018}.
This idea of assigning the local estimator to basis functions instead of elements has also been studied in the context of finite elements in \cite{BM1987,MNS2003}.

In this work we want to address adaptivity under the assumptions that
the allowed bases are subsets of B--splines of different levels (hierarchical splines) and
that the refinement is done on functions rather than elements.
Under this setting we study the question of what are the simplest but
yet theoretically adequate spaces to successfully develop a theory of
adaptivity.

To be more precise, let \calB be the set of B--splines of all levels.
We want to find a family $\textgoth{F}$ of subsets of \calB, and a
procedure \Call{Refine}{} such that they simultaneously satisfy 
the following properties.
\begin{property}[About \Call{Refine}{}]\label{p:refine}
Given $\calH\in\textgoth{F}$ and $\calM\subset\calH$,
the procedure \Call{Refine}{} returns $\calH_*\in\textgoth{F}$ where:
\begin{enumerate}[(i)]
\item\label{p:refine:res}
$\calH_*$ has more resolution in the places indicated 
by \calM;
\item\label{p:refine:complexity}
when used in the adaptive loop \eqref{e:adaptive_loop} 
there is $C$ (independent of $R$ and $\calM_r$'s) such that
 $\#\calH_R-\#\calH_0 \leq C\sum_{r=0}^{R-1}\#\calM_r$;
\item \label{p:refine:simple}
 It is simple to implement computationally.
\end{enumerate}
\end{property}

\begin{property}[About $\textgoth{F}$]\label{p:generators}
Given $g\in\Z^+$, $\textgoth{F}$ satisfies the following:
\begin{enumerate}[(i)]
\item \pedro{The spaces generated by $\calH \in \textgoth{F}$ possess} good approximation properties in terms of the number of 
degrees of freedoms;
\item\label{p:generators:li} 
 if $\calH\in\textgoth{F}$ then \calH is linearly independent; and
\item\label{p:generators:overlap}
if $\calH\in\textgoth{F}$ for any two functions in \calH that overlap 
their level difference is at most $g$.
\end{enumerate}
\end{property}

Requirement \ref{p:generators}\ref{p:generators:li} is important
for the \Call{Solve}{} in the adaptive loops as well as the design
of estimators.
\pedro{The constraint on the overlap of functions from different levels given by \ref{p:generators}\ref{p:generators:overlap} is a technical requirement for the proof of a contraction property of adaptive algorithms; at a certain point an inverse inequality is required, which cannot be bounded with a uniform constant unless this assumption is met. When \ref{p:generators}\ref{p:generators:overlap} holds, we say that the \emph{gap} is bounded by $g$.}
The complexity bound \ref{p:refine}\ref{p:refine:complexity} is key for the
optimality results.

The motivational guideline to simultaneously satisfy both sets of properties is to start by considering all possible generators (hierarchical generators) obtainable by refinements; a concept that needs a rigorous definition.
%%
%Which in turn leads to define a rigorous concept of refinement.
%
%We will say that an element $\calH_*\in \textgoth{F}$ is a refinement of $\calH\in \textgoth{F}$
We define the refinement relation in $\textgoth{F}$ as a set inclusion of certain function sets, the \emph{lineages}, which are associated in a one-to-one fashion with the generators in $\textgoth{F}$ (see Lemma~\ref{l:bijection}).
%%
%We show that the lineages keep a one to one relation with the generators 
%[Lemma~\ref{l:bijection}]
Thus, the refinement relation induces a partial order on the hierarchical spaces allowing us to rigorously and simply state questions such as 
\begin{quote}
Given a hierarchical basis \calH, which is the smallest refinement of \calH that 
is a basis, has refined $\vphi\in\calH$ and its gap is bounded by
a given number $g$?
\end{quote}
 
We start with some technical preliminaries where we set
the notation and language to conveniently handle the
ancestry and overlapping relation in terms of the multilevel
tensor index of the B--splines (sections \ref{s:preliminaries} and
\ref{s:multilevel}).
We have two notations, one that keeps track of the indexes which is useful for 
the computational implementation of the actual algorithms and another one which 
bounds ``distances'' which is useful for the analysis of complexity.
Section \ref{s:generators} introduces the concept of lineage which
is key for the definition of refinement in Section~\ref{s:refinement}.
In general the generators are not linearly independent, thus in
Section~\ref{s:linear_independence} we discuss how to restrict the
refinement process to yield bases.
Section \ref{s:gap} deals with the important matter of restricting the
refinement to yield bases with a uniformly bounded gap.
Finally in Section~\ref{s:complexity} we show that the refinement
process we have presented, which yields a refined bases with a given bound on
their gaps, satisfies the required complexity bound from Poperty~\ref{p:refine}\ref{p:refine:complexity} when used in the adaptive loop \eqref{e:adaptive_loop}. More precisely, the dimension of each hierarchical space is linearly bounded by the history of marked functions.

\pedro{Some auxiliary results have been collected in an Appendix (Section~\ref{sec:appendix}) in order not to interrupt the flow of ideas in the core of this article.}

\pedro{
The main contributions of the work are:
\begin{enumerate}
\item Characterize refinement as set inclusion (of associated \emph{lineages}) which allows to rigorously talk about the smallest refinement that satisfies certain property.
\item Provide constructive algorithms with rigorous proofs that they perform the required tasks.
\item A thorough analysis of the different ingredients, which are separated showing what implications and properties are linked to each other.
\end{enumerate}

Hierarchical spline spaces were introduced in~\cite{VGJS2011,Kraft1997}, and some modifications were introduced in~\cite{KGS2003,BG2016}. 
Those definitions, except the one in~\cite{KGS2003}, are stated in terms of sequences of subdomains.
Recently, optimality of adaptive isogeometric methods has been proved for elliptic problems~\cite{BG2017,GHP2017}. Those results are based on residual-type a posteriori estimators associated to elements or cells, where one notices some difficulties in the handling on the definition of hierarchical spaces in terms of a domain plus the need to have admissible meshes (gap bounded by one), both works are element-oriented. 
The function-based refinement can be traced to ideas of Krysl~\cite{KGS2003}, further develop by Garau and Buffa~\cite{BG2016}, where they focus on a positive partition of unity.

Our main goal was to rethink a definition of hierarchical spaces with the concept of refinement based on basis functions, that would yield the simplest language to work, theoretically as well as practically.
}

\section{Preliminaries}\label{s:preliminaries}
This section is mainly intended to set the general notation and introduce basic 
concepts.
We will work in the euclidean space $\R^d$ ($d \in \Z^+$).
The tensor product structure intrinsic to the multivariate spline spaces 
leads to many concepts being isomorphic to the integer lattice $\Z^d$, hence
the following notation is convenient.
Given two integers $j$ and $k$ we let 
$[j:k] := \{ l \in \Z : j \leq l \leq k \}$ denote
the section of integers bounded by them.
A multi-index $\vec{j}$ is an element of the integer lattice $\Z^d$,
with its $i$-th component denoted by $j_i$.
Given $\vec{j}$ and $\vec{k}$ in $\Z^d$, we let 
$[\vec{j}:\vec{k}] := \times_{i=1}^d [j_i:k_i]$
denote the lattice box bounded by them.
% and $\vec{j}\pm\vec{k} := [\vec{j}-\vec{k}:\vec{j}+\vec{k}]$.
%%
For $\vec x$ and $\vec y$ in $\R^d$ we let
$\dist(\vec x,\vec y)=\|\vec x-\vec y\|_{\infty} = 
\max_{1\leq i\leq d} |x_i-y_i|$,
so that, for instance, $\diam[\vec{j}-\vec{k}:\vec{j}+\vec{k}] =2 \|\vec k\|_{\infty}$.

%% And $B_r(\vec x)=\{\vec y :\dist(\vec x,\vec y)<r\}$ will
%% denote the open ball centered at $\vec x$.
%%
Given a finite set $A$, its size or number of elements is denoted 
by $\# A$, whence
$\#[\vec{j}:\vec{k}] = \Pi_{i=1}^d (k_i-j_i+1)$.
When a scalar value appears in the place where a vector value is expected
it means the vector where each component has the value of the scalar.
For example if $\vec i\in\Z^3$ then $[\vec i: 2] = [\vec i: (2,2,2)]$.
When an operation/relation (like inequality) is applied to vectors it means that it is applied to each of their components, i.e., $\vec i < \vec j$ if and only if $i_k<j_k$ for $k\in[1:d]$. 

Finally, the notation $C = A \dcup B$ means that $C = A \cup B$ 
and $A \cap B = \emptyset$, and $\lfloor x\rfloor$ stands for the floor of a 
real number $x$, i.e. the largest integer smaller than or equal to $x$.
\subsection{Index and Box preserving functions}\label{subsec:index_functions}
Whenever a function $P$ goes from $\Z^d$ into $\Z^d$ we call it
an \emph{index function}.

Let $n\ge 2$ be a fixed integer.  
For any $m\in\Z$ we define the index functions
\begin{equation}\label{eq:index_functions_M_D} 
\ipro{m}{\vec i} := n \vec i + m \quad\text{ and }\quad \fdiv{m}{\vec i} := 
\left\lfloor \frac{\vec i-m}{n} \right\rfloor. 
\end{equation} 

For $k\in\Z^+$ the function $\operatorname{D}_m^k$ denotes the $k$-th iterate of 
$\operatorname{D}_m$, that is $\operatorname{D}_m^1=\operatorname{D}_m$ and 
$\fdiv[k]{m}{\cdot}=\fdiv{m}{\fdiv[k-1]{m}{\cdot}}$. For completeness
$\operatorname{D}_m^0$ is the identity function.
Similarly the same definition is considered for  
$\operatorname{M}_m^k$.
Clearly, $\fdiv[k]{m}{\ipro[k]{m}{\vec i}} = \vec i$, but it is not always true 
that $\ipro[k]{m}{\fdiv[k]{m}{\vec i}} = \vec i$. However, the following result, which is an immediate consequence of Lemma~\ref{l:inequalities} in the Appendix,
holds.

\begin{lemma}\label{l:ineq_iff}
	Let $m,m'\in\Z$ and $k\in\Z^+$, then
	$\ipro[k]{m}{\vec j} \leq \vec i \leq \ipro[k]{m'}{\vec j}$
	if and only if 
	$\fdiv[k]{m'-(n-1)}{\vec i} \leq \vec j  \leq \fdiv[k]{m}{\vec i}$.
\end{lemma}
%\begin{proof}
%	Immediate from Lemma~\ref{l:inequalities}.
%\end{proof}

For any fixed $g\in\Z^+$ and $p\in\Z$, let us define two more index functions
\begin{equation}\label{eq:index_functions_L_R}
\operatorname{L}(\vec i) :=\fdiv[g]{0}{\vec i} - p \quad\text{ and }\quad 
\operatorname{R}(\vec i) :=\fdiv[g]{0}{\vec i + p}.
\end{equation}  
and their $k$-th iterates $\operatorname{L}^k(\vec i)$ and $\operatorname{R}^k(\vec i)$ for $k\in\Z^+$.
%Once more, for any $k\in\Z^+$ the functions 
%$\operatorname{R}^k(\vec i)$ and $\operatorname{L}^k(\vec i)$ denote the $k$-th 
%iterate of $\operatorname{R}$ and $\operatorname{L}$, respectively.

A function $F$ given by $F(\vec i)=[P(\vec i):Q(\vec i)]$
where $P,Q$ are index functions is called a \emph{box function}, 
and it is called \emph{box preserving} if
$F[[\vec i:\vec j]] = [P(\vec i):Q(\vec j)]$.

\begin{lemma}\label{l:box_preserving1}
	Let $k\in\Z^+$. If $m'-m \geq n-1$ then 
	$F(\vec i) = [\ipro[k]{m}{\vec i}:\ipro[k]{m'}{\vec i}]$ is box preserving.
	And if $m-m'\geq 0$ then $F(\vec i) = (\fdiv[k]{m}{\vec 
		i}:\fdiv[k]{m'}{\vec i})$ is box 
	preserving.
\end{lemma}

The proof of Lemma~\ref{l:box_preserving1} %and Lemma~\ref{l:box_preservingLR} 
is immediate from Lemma~\ref{lem:decouples} and 
Corollary~\ref{lem:box_preserving_iteration}.

\seba{%
\subsection{B-spline basis}
\subsubsection{One dimensional splines}\label{s:oned_spline}
Let $[a,b]$ be an interval and
$\Delta=\{x_j\}_1^n$ with $a=x_0<x_1<\dots<x_n<x_{n+1}=b$
a partition into $n+1$ $1$-cells (subintervals)
$I_j=[x_j,x_{j+1})$ for $j\in[0:n-1]$ and $I_n=[x_n,x_{n+1}]$.
Let $m$ be a positive integer and 
$\vec M\in\Z^n$ with $1\leq \vec M < m$.
The spline space 
$\V = \V(\Delta, m, \vec M)$ is the space of piecewise polynomials of order
$m$ that are $C^{m-M_j-1}$ at $x_j$ for $j\in[1:n]$.
The elements of $\Delta$ are the interior knots and $\vec M$ 
is the interior multiplicity vector.

Given an extended partition associated with $\V(\Delta, m, \vec M)$
(dictated by the multiplicity vector with an appropriate selection of 
additional end knots) there exists a constructive
process that produces a basis of $\V$ known as the B-spline basis
(see \cite[Theorem 4.9]{Sch2007}).
It is well known (\cite{Sch2007}) that these basis functions have 
minimal support, 
are locally linear independent, 
non-negative and form a partition of unity.
}

\subsubsection{Tensor product splines}\label{s:dd_spline}
If for each $k\in[1:d]$ an interval $[a_k,b_k]$, a
knot partition $\Delta_k$, a positive integer $m_k$
and multiplicity vector $\vec M_k\in\Z^{n_k}$ are provided,
then following the process of Section~\ref{s:oned_spline}
we can define $d$ one-dimensional B--spline bases $\calB_k$, $k\in[1:d]$.
% The $n$-th one-dimensional cell of the $k$-th partition is denoted by $I_n^k$.

It can be shown \cite{Sch2007} that the set
$\calB = \{\phi = \otimes_{k=1}^{d} \phi_k : \phi_k\in\calB_k\}$
is a family of linearly independent functions over the box 
$\Omega=\times_{k=1}^{d} [a_k,b_k]$, 
which are also non-negative and form a partition of unity. 
The function $\otimes_{k=1}^{d} \phi_k : \Omega \to \R$ is the tensor product of the univariate functions $\phi_k$, i.e.,
$\big(\otimes_{k=1}^{d} \phi_k\big) (\vec{y}) = \Pi_{k=1}^d \phi_k(y_k)$.

Under this scenario we define the tensor product spline
space 
\[ 
\V = \V(\{\Delta_k\}, \{m_k\},\{\vec M_k\})
\]
as the set of linear combinations of the elements of $\calB$.

%%%%%%%%%%%%%%%%%%%%%%%%%%%%%%%%%%%%%%%%%%%%%%%%%%%%%%%%%%%%%%%%%%
\section{Multilevel B-splines}\label{s:multilevel}
To each non negative integer \l we want to associate a set of B-splines 
where \l indicates the level of resolution.
Roughly speaking, \l is a measure of the knot density.
In this work for the sake of clarity we restrict ourselves to the subclass 
of spline spaces where the knots of level $\l+1$ are obtained by adding 
$s\in\Z^+$ knots uniformly distributed between the knots of level \l.
And the knots of level $0$ are $\{0,1\}^d$.
Thus we take the domain $\Omega$ to be the unit cube $[0,1]^d$ and in 
going from level \l to $\l+1$ each subinterval is divided into
$n:=s+1$ equal-length subintervals of level $\l+1$.

We also consider maximum interior regularity for the spline spaces.
These restrictions simplify the notation to better concentrate in the new
concepts and ideas introduced in this work but there is no essential impediment 
to extend these ideas to non-uniform and less regular cases.
%(cf. Section \ref{s:future_work}).

\subsection{Multilevel cells}
From now on we fix the value of $s \in \Z^+$ and let $n = s+1$.
For $\l \in \Z^+_0$, $i \in [0,n^\l-2]$ let 
$I^\l_i :=[\frac{i}{n^\l},\frac{i+1}{n^\l})$ 
and $I^\l_{n^\ell-1} :=[1-\frac{1}{n^\l},1]$.
\pedro{And for a given $\vec{i} \in [0,n^\l-1]^d = [\ipro[\l]{0}{\vec0}:\ipro[\l]{s}{\vec 0}]$ we define 
$I^\l_{\vec{i}} := \times_{k=1}^d I^\l_{i_k}$ the $\vec{i}$-th $n$-adic cell 
of level \l in dimension $d$.
Moreover, we define $\calI^\l := \{I^\l_{\vec{i}}: \vec{i} \in 
[\ipro[\l]{0}{\vec0}:\ipro[\l]{s}{\vec 0}]\}$ as the set of \emph{cells of 
level $\l$} and $\calI:=\dcup_{\l=0}^\infty \calI^\l$ \emph{the multilevel 
cells}.}
\begin{definition}[Level and index of a cell]
 Given $I\in\calI$ the level of $I$, denoted
 by $\l_{I}$, and the index of $I$, denoted by $\vec{i}_I$, are the unique 
integer and index respectively such that $I^{\l_{I}}_{\vec{i}_I} = I$.
\end{definition}
\begin{definition}[Box of cells]\label{d:box_cells}
	Given $\l\in\Z^+_0$, $\vec{i}$ and $\vec{j}$ in $\Z^d$ with 
	\pedro{$0\le \vec{i} \leq \vec{j} \le n^\l-1$}, we define 
	$\calI^\l[\vec{i}:\vec{j}] :=  
	\{  I^\l_{\vec{k}} : \vec{k}\in[\vec{i}:\vec{j}]\}$
	as the box of cells of level \l bounded by the corners 
	$\vec{i}$ and $\vec{j}$.
\end{definition}

\begin{definition}[Children of a cell]\label{d:cell_ch}
Given $I=I^\l_{\vec{i}}\in\calI$, the the set $\ch I$ of children of $I$ is defined as
$\ch{I^\l_{\vec{i}}} := \calI^{\l+1}[\ipro{0}{\vec{i}}:\ipro{s}{\vec{i}}]$.
And if $\calJ \subset \calI$ then $\ch{\calJ}=\cup_{I\in\calJ}{\ch{I}}$.
\end{definition}

	\begin{remark}
		\label{R:cell_ch}
		Note that $J \in \ch{I}$ iff $\ell_J = \ell_I+1$ and $J\subset I$, and 
moreover $I = \bigcup_{J\in\ch{I}}J$.
	\end{remark}

\begin{definition}[Descendants and ancestors]
Given $I \in \calI$ and $k\in\Z^+$ we define the set of its $k$-descendants 
as the set $\ch[k]{I}$, resulting from $k$ successive applications of the 
children operator.
We also define the $k$-ancestors of $I$ by
$\ch[-k]{I}=\{J\in\calI: I\in \ch[k]{J}\}$, and
$\ch[0]{I}=\{I\}$.
Finally, if $\calJ \subset \calI$ then $\ch[k]{\calJ}=\cup_{I\in\calJ}{\ch[k]{I}}$.
\end{definition}
\begin{lemma}[Box ancestry]\label{l:cell_hie}
For $k\in\Z^+$ and $\vec{i},\vec{j}\in\Z^d$ with $\vec{i}\leq\vec{j}$
\begin{enumerate}[(i)]
\item $\ch[k]{\calI^{\l}[\vec{i}:\vec{j}]} = 
    \calI^{\l+k}[\ipro[k]{0}{\vec{i}}:\ipro[k]{s}{\vec{j}}]$
 \item $\ch[-k]{\calI^{\l}[\vec{i}:\vec{j}]} = 
 \calI^{\l-k}[\fdiv[k]{0}{\vec{i}}:\fdiv[k]{0}{\vec{j}}]$   
\end{enumerate}
\end{lemma}
\begin{proof}
Let $k=1$ using Definition \ref{d:cell_ch} we have
$\ch{\calI^{\l}[\vec{i}:\vec{j}]} = 
\cup_{\vec n \in [\vec i: \vec j]} 
\calI^{\l+1}[\ipro{0}{\vec{n}}:\ipro{s}{\vec{n}}].$
Let $F(\vec{n}) = [\ipro{0}{\vec{n}}:\ipro{s}{\vec{n}}]$, from
Lemma \ref{l:box_preserving1} it follows that $F$ is box
preserving so we have that
$\ch{\calI^{\l}[\vec{i}:\vec{j}]} = 
\calI^{\l+1}[\ipro{0}{\vec i}:\ipro{s}{\vec j}]$.
For $k>1$ the proof follows by induction using the result we have just shown.

For (ii) observe that 
$\ch[-k]{\calI^{\l}[\vec{i}:\vec{j}]} =
\cup_{\vec n \in [\vec i: \vec j]}
\ch[-k]{I^{\l}_{\vec n}}$. From here we have that
$J\in\ch[-k]{\calI^{\l}[\vec{i}:\vec{j}]}$ if and only if
$J=I^{\l-k}_{\vec r}$ and there is $\vec n \in [\vec{i}:\vec{j}]$ such
that  $I^{\l}_{\vec n} \in \ch[k]{I^{\l-k}_{\vec r}}$.
And the latter happens if and only if
$\vec n \in [\ipro[k]{0}{\vec{r}}:\ipro[k]{s}{\vec{r}}]$.
From Lemma~\ref{l:ineq_iff} this is equivalent to 
$\vec{r} = \fdiv[k]{0}{\vec{n}}$.
Thus we conclude that 
$\ch[-k]{\calI^{\l}[\vec{i}:\vec{j}]}=
\cup_{\vec{r} \in [\fdiv[k]{0}{\vec{i}}:\fdiv[k]{0}{\vec{j}}]} I^{\l-k}_{\vec 
r} = \calI^{\l-k}[\fdiv[k]{0}{\vec{i}}:\fdiv[k]{0}{\vec{j}}]$.
\end{proof}

\begin{comment}
\begin{definition}[Partition of a cell]
Given $I\in\calI$ a partition of $I$ is a set $\mathcal{P}\subset\calI$
such that \marce{no element of $\mathcal{P}$ is a descendant of another element 
of $\mathcal{P}$} and $I=\cup_{J\in\mathcal{P}} J$.
\end{definition}

Trivially, %Due to Remark~\ref{R:cell_ch}, 
	$\ch[]{I}$ is a partition of $I$, and if $\calJ$ is a partition of $I$, any 
set obtained by replacing cells in $\calJ$ by their children is again a 
partition of $I$. \marce{Moreover, if $\calJ$ and $\calI$ are two disjoint 
subsets of $\calI^\l$ then
 $\ch[k]{\calJ}\cap\ch[k]{\calI}=\emptyset$.}
\end{comment}

\subsection{Multilevel B-splines}%%
\label{Bsplines}
Let $V^{(\l,m,d)}(\Omega)$ be the space of tensor product splines of order $m$ 
globally $C^{m-2}$ subordinated to the $n$-adic partition of level \l of $\Omega$
(see Section~\ref{s:dd_spline}).
From this point on, $m$ and $d$ are fixed unless explicitly stated so we drop 
them from the notation and write for example $V^\l$ instead of 
$V^{(\l,m,d)}(\Omega)$. It is also convinient to introduce the number
$p:=m-1$ which is the \emph{degree} of the B--splines.

The ``master'' B--spline of order $m$ is
\begin{equation}
Q(x) := m \sum_{j=0}^m (-1)^j {m \choose j} (x-j)_+^{m-1}, 
\end{equation}
where $f_+$ stands for the positive part of $f$.
The function $Q$ is $C^{m-2}$, positive in $(0,m)$ with support equal to $[0,m]$. 
For $\l\geq 0$ and $i\in\Z$ let
$\vphi^\l_i(x) := Q(n^\l x - i)$ and if $\vec{i} \in \Z^d$ define
$\vphi^{\l}_{\vec{i}}(\vec x) := \Pi_{k=1}^d \vphi^\l_{i_k}(x_k)$.
With these definitions we make the 
B--splines sets clearly isomorphic to
integer lattices. 
More precisely, it can be shown \cite{Sch2007} that for each \l a normalized 
tensor product B-spline basis (of level $\l$) on $\Omega$ is
\begin{equation}\label{eq:Bl_def}
\calB^\l = \{ \vphi^\l_{\vec{i}} : 
\vec{i} \in [-p:\ipro[\l]{s}{0}] \}.
\end{equation}
\pedro{These are the so-called \emph{cardinal} B-splines, which correspond to extending the uniform partitions beyond the boundaries of $\Omega$.
If we considered the so-called \emph{interpolatory} B-splines, which correspond to the \emph{open-knot} vector, the only difference would be that not all the basis functions are dilations and translations of the same master B-spline $Q$, but of a finite number of master functions. All what follows remains valid.}
Trivially from the definition, the sets $\calB^\l$ are pairwise
disjoint. Let $\calB := \dcup_{\l=0}^{\infty} \calB^\l$ be
the set of \emph{multilevel B-splines}.
\begin{definition}[Level and index of a B-spline]
 Given $\vphi\in\calB$, the level of $\vphi$, denoted
 by $\l_{\vphi}$, and the index of $\vphi$, denoted
 by $\vec{i}_\vphi$, are the unique integer 
 and index, respectively, such that
 $\vphi^{\l_{\vphi}}_{\vec{i}_\vphi} = \vphi$.
\end{definition}

\begin{definition}[Box of B--splines]\label{d:box_spline}
 Given $\l\in\Z^+_0$, $\vec{i}$ and $\vec{j}$ in $\Z^d$ with 
 $\vec{i} \leq \vec{j}$ we define 
 $\calB^\l[\vec{i}:\vec{j}] :=  
 \{  \vphi^\l_{\vec{k}}: \vec{k}\in[\vec{i}:\vec{j}]\}$
 as the box of B--splines of level \l bounded by the corners 
 $\vec{i}$ and $\vec{j}$.
\end{definition}

\begin{definition}[Children of a B-spline]\label{d:bsp_ch}
Given $\vphi=\vphi^\l_{\vec{i}}\in\calB$ the set $\ch \vphi$ of children of $\vphi$ 
is defined as
$\ch{\vphi^\l_{\vec{i}}} := \calB^{\l+1}[\ipro{0}{\vec{i}}:\ipro{sm}{\vec{i}}]$.
Also, if $\calF \subset \calB$, we define 
$\ch{\calF} := \cup_{\vphi\in\calF}\ch{\vphi}$.
\end{definition}

This definition is motivated by the following Lemma.

\begin{lemma}[Children properties]\label{l:cp}
\begin{enumerate}[(i)]
\item \label{l:cp:ch_coef} There exist $c_{\vec{k}}\in\R^+$ for $\vec{k} 
\in [\vec 0:sm]$ such that 
$\vphi^\l_{\vec{i}}(x) = 
\sum_{\vec{k} \in [\vec 0:sm]} c_{\vec{k}} \vphi^{\l+1}_{n\vec{i}+\vec{k}}$, 
for every $\l\in\Z^+_0$ and $\vec i \in \Z$.

%%%
\item\label{l:cp:equiv} Let $\psi\in\calB^{\l}$ and 
$\psi = \sum_{\vphi\in\calB^{\l+1}} \alpha_{\vphi} \vphi$
be its unique expansion in $\calB^{\l+1}$
then $\vphi\in \ch{\psi}$ if and only if $\alpha_{\vphi}>0$.

\end{enumerate}
\end{lemma}
\seba{%
\begin{proof}
Part \ref{l:cp:ch_coef} follows from the standard recurrence relation 
for B--splines of consecutive order (see \cite[p.90]{deB2001}) 
and the fact that the result holds for B--splines of order $1$.
Part~\ref{l:cp:equiv} follows easily from~\ref{l:cp:ch_coef} and the
fact that all the B--splines from a fixed level are linearly independent.
\end{proof}
}%

\pedro{
\begin{remark} This parent-children relationship holds with the same coefficients $c_{\vec{k}}$ for all functions of all levels in the case of cardinal B-splines. In the case of interpolatory B-splines there is a finite number of different situations, corresponding to the cases when the support of the involved basis functions touch the boundary of $\Omega$. But this is not essential for the discussion of this article.
\end{remark}
}

\begin{definition}[Descendants and ancestors]
    Given $\vphi$ in $\calB$ and $k\in\Z^+$ we define the set of its 
$k$-descendants
    as the set $\ch[k]{\vphi}$, resulting from $k$ successive applications of 
the children operator.
    We also define $k$-ancestors of $\vphi$ by 
    $\ch[-k]{\vphi}=\{\psi\in\calB: \vphi\in \ch[k]{\psi}\}$,
    and $\ch[0]{\vphi}=\{\vphi\}$.
    Finally, if $\calF \subset \calB$ then 
$\ch[k]{\calF}=\cup_{\vphi\in\calF}{\ch[k]{\vphi}}$.
\end{definition}

\begin{lemma}[Box ancestry]\label{l:bspline_hie}
    For $\l\geq0$, $\vec i,\vec j \in\Z^d$ and $k\in\Z^+$ we have that
    \begin{enumerate}[(i)]
        \item $\ch[k]{\calB^{\l}[\vec{i}:\vec{j}]} = 
        \calB^{\l+k}[\ipro[k]{0}{\vec{i}}:\ipro[k]{sm}{\vec{j}}]$
        \item $\ch[-k]{\calB^{\l}[\vec{i}:\vec{j}]} = 
        \calB^{\l-k}[\fdiv[k]{sp}{\vec{i}}:\fdiv[k]{0}{\vec{j}}]$   
    \end{enumerate}
\end{lemma}
%%%%%%%%%%%%%%%
\begin{proof}
Let $k=1$, using Definition \ref{d:bsp_ch} we have
$\ch{\calB^{\l}[\vec{i}:\vec{j}]} = 
\cup_{\vec n \in [\vec i: \vec j]} 
\calB^{\l+1}[\ipro{0}{\vec{n}}:\ipro{sm}{\vec{n}}]$.
From Lemma \ref{l:box_preserving1} the function 
$F(\vec{n}) = [\ipro{0}{\vec{n}}:\ipro{sm}{\vec{n}}]$ is box
preserving, then it follows that 
$\ch{\calB^{\l}[\vec{i}:\vec{j}]} = 
\calB^{\l+1}[\ipro{0}{\vec i}:\ipro{sm}{\vec j}]$.
The rest of the proof follows by induction on $k$.

To show (ii) observe that 
$\ch[-k]{\calB^{\l}[\vec{i}:\vec{j}]} =
\cup_{\vec n \in [\vec i: \vec j]}
\ch[-k]{\vphi^{\l}_{\vec n}}$. From here we have that
$\vphi^{\l-k}_{\vec r}\in\ch[-k]{\calB^{\l}[\vec{i}:\vec{j}]}$ 
if and only if there is 
$\vec n \in [\vec{i}:\vec{j}]$ such
that  $\vphi^{\l}_{\vec n} \in \ch[k]{\vphi^{\l-k}_{\vec r}}$.
And from part (i) of the result, the latter happens if and only if
$\vec n \in [\ipro[k]{0}{\vec{r}}:\ipro[k]{sm}{\vec{r}}]$
which by Lemma~\ref{l:ineq_iff} is equivalent  to 
$\fdiv[k]{sm-s}{\vec{n}} \leq \vec{r} \leq \fdiv[k]{0}{\vec{n}}$.
We thus get
$\ch[-k]{\calB^{\l}[\vec{i}:\vec{j}]} =
\cup_{\vec n \in [\vec i: \vec j]} 
\calB^{\l-k}(\fdiv[k]{s(m-1)}{\vec{n}}:\fdiv[k]{0}{\vec{n}})$.
Now $F(\vec n) = (\fdiv[k]{s(m-1)}{\vec{n}}:\fdiv[k]{0}{\vec{n}})$
is a box preserving operator, thus
$\ch[-k]{\calB^{\l}[\vec{i}:\vec{j}]} =
\calB^{\l-k}(\fdiv[k]{s(m-1)}{\vec{i}}:\fdiv[k]{0}{\vec{j}})$.
\end{proof}  

For the complexity results it will be useful to have a notion of
a ``ball of functions'' which are related to a scaled 
comparison of the indexes as given in the following definition.
\begin{definition}[Oriented distance]
Let $\vphi_1$ and $\vphi_2$ in \calB, we define
$$\rho(\vphi_1, \vphi_2) := 
\frac{\vec{i}_{\vphi_1}}{n^{\l_{\vphi_1}-\l_{\vphi_2}}}-{\vec{i}_{\vphi_2}}.$$
\end{definition}

Observe that $\rho$ is not symmetric, in fact
$n^{\l_{\vphi_1}} \rho(\vphi_1, \vphi_2) = 
- n^{\l_{\vphi_2}} \rho(\vphi_2,\vphi_1)$. 
Moreover it satisfies the following analogous of 
the triangle inequality, whose 
proof is easily obtained by induction.
\begin{lemma}[Weighted triangular equality]\label{l:rho_triangle}
 Let $\vphi_0,\dots\vphi_j \in \calB$ then
 $$\rho(\vphi_j, \vphi_0) = \sum_{i=0}^{j-1} 
\frac{1}{n^{\l_{\vphi_i}-\l_{\vphi_0}}} \rho(\vphi_{i+1}, \vphi_i).$$
\end{lemma}

We can use $\rho$ to express the descendants of a B--spline
in the sense of the next Lemma. This will be useful when analyzing implications of the refinement algorithm in Section~\ref{S:refinement with gap constraint}.
\begin{lemma}[\seba{Distance to descendants}]\label{l:desc_rho}
Let $\eta\in\calB^\ell$ and $k\in\Z^+$, then
\begin{equation*}
\begin{split}
\ch[k]{\eta} &= \{\psi\in\calB^{\l+k} : 0\leq \rho(\psi,\eta)\leq m 
(1-\tfrac{1}{n^k}) \} \\
&= \{\psi\in\calB^{\l+k} : - m(n^k-1) \le \rho(\eta,\psi)\le 0\}
\\
&\subset \{\psi\in\calB^{\l+k} : 0\leq \rho(\psi,\eta)\leq m \}
= \{\psi\in\calB^{\l+k} : -mn^k \le  \rho(\eta,\psi)\leq 0\}
.
\end{split}
\end{equation*}
\end{lemma}
\begin{proof}
	From Lemma~\ref{l:bspline_hie}, $\psi \in \ch[k]{\eta}$ iff $\psi \in 
\calB^{\ell+k}(M_0^k(\vec i_\eta):M_{sm}^k(\vec i_\eta))$, i.e., 
	$M_0^k(\vec i_\eta) \le \vec i_\psi \le M_{sm}^k(\vec i_\eta)$, or
	$ 0 \le \vec i_\psi - M_0^k(\vec i_\eta) \le M_{sm}^k(\vec i_\eta) - 
M_0^k(\vec i_\eta)$.
	Since $M_0^k(\vec i_\eta) = n^k \vec i_\eta$ and $M_{sm}^k(\vec i_\eta) - 
M_0^k(\vec i_\eta) = m (n^k-1)$, we have, after dividing by $n^k$, that $\psi 
\in \ch[k]{\eta}$ iff
	\[
	0 \le \frac{\vec i_\psi}{n^k} - \vec i_\eta \le m ( 1 - \frac{1}{n^k}) < m,
	\]
	and the assertion follows.
\end{proof}

\begin{definition}[Ball of functions]\label{d:ballofsplines}
Let $\vphi \in \calB$, $D\in\R^+$ and $k\in\Z$ then we 
define
$\ball{\vphi}{D}{k} :=
\{\psi\in\calB^{\l_\vphi+k} :
|\rho(\vphi,\psi)|\leq D \}$
\end{definition}

\begin{lemma}[Uniform bound by level]\label{l:radius}
Let $\vphi \in \calB$, $D\in\R^+$ and $k\in\Z$ then
$\#\ball{\vphi}{D}{k} \leq (2D+1)^d$.
\end{lemma}

\begin{proof}
By definition
\begin{align*}
\#\ball{\vphi}{D}{k} 
&= \# \left\{ \psi \in \calB^{\ell_\vphi+k} : 
\left| \frac{\vec i_\vphi}{n^{\ell_\vphi-\ell_\psi}} - \vec i_\psi \right|
   \le D \right\}
\\
&\leq \#\left\{  \vec i \in \Z^d : 
\left| \frac{\vec i_\vphi}{n^{-k}}- \vec i \right|
\le D \right\}
\\
&\le \# \left\{ \vec i \in \Z^d : -D + \vec i_\vphi n^k \le \vec i \le D + \vec 
i_\vphi n^k \right\},
\end{align*}
which immediately implies the claim.
\end{proof}

\subsection{B--splines overlapping}
This section considers the overlapping among B--splines of different levels.
\begin{definition}[Cells supporting a B-spline]\label{d:cellsupp}
    Given $\vphi_{\vec i}^\l\in\calB$, we define the set
    $\bbI(\vphi_{\vec i}^\l) := 
    \calI^\l[\vec i:\vec i + p]$
    as its cell support.
\end{definition}

%\begin{lemma}
%If $\vphi\in\calB$ then $I\in\bbI(\vphi)$ implies
%$ \vphi>0$ in $\mathring I$.
%\end{lemma}

This definition is motivated from the fact that $I\in\bbI(\vphi)$ if and only if $ \vphi>0$ in $\mathring I$.
Moreover $\supp \vphi_{\vec i}^\ell = \overline{\cup_{\vec 
k \in [\vec i: \vec i+p]} I_{\vec k}^\ell}$.
The proof of this fact is immediate from the definition of $\vphi_{\vec i}^\l$ (see beginning of Section~\ref{Bsplines}).
%It is worth noting that in the case of maximum regularity splines, 
%which is the case considered throughout this article, 
%$\calI^\ell [\vec i: \vec i+p]$ is the set of the cells of level $\ell$ contained in the 
%support of $\vphi_{\vec i}^\ell$, i.e., $\supp \vphi_{\vec i}^\ell = \cup_{\vec 
%k \in [\vec i: \vec i+p]} I_{\vec k}^\ell$.

We insist on working algebraically with sets of indices, 
because it is directly translated into the implementation, and easier to check 
for correctness.

\begin{definition}[Cells overlapping a B-spline]\label{d:celloverlap}
    Given $k\in\Z$ and $\vphi\in\calB$ we define
    $\bbI^k(\vphi) := \ch[k]{\bbI(\vphi)}$. If
    $\calF\subset\calB$ then
    $\bbI^k(\calF) := \cup_{\vphi\in\calF} \bbI^k(\vphi)$.
\end{definition}

Note that $\bbI^k(\varphi) \subset \calI^{\ell_\vphi+k}$ and in the case of 
maximum regularity splines, $\bbI^k(\vphi)$ is the set of cells of level 
$k+\ell_\vphi$ which overlap with the support of $\vphi$.
Note also that the \emph{children} operator $\ch[]{}$ is defined from 
$\calB$ to $\calB$ (and from $\calL$ to $\calL$, see Section 
\ref{s:generators}). In the next lemma we explore on 
this.
	
	\begin{lemma}[Interchange of $\bbI$ and $\ch{}$]\label{L:children}
		For $\vphi \in \calB$ and $k \in \Z^+_0$ we have
		\begin{enumerate}[(i)]
			\item $\bbI(\ch[k]{\vphi}) = \ch[k]{\bbI(\vphi)} = \bbI^k(\vphi)$
			\item $\bbI(\ch[-k]{\vphi}) \supset \ch[-k]{\bbI(\vphi)} = 
\bbI^{-k}(\vphi)$
		\end{enumerate}
	\end{lemma}

\begin{proof}
	Let $\vphi = \vphi_{\vec i}^\ell \in \calB$ and $k \in \Z_0^+$. 
	On the one hand, from Lemma~\ref{l:bspline_hie}, 
	$\ch[k]{\vphi_{\vec i}^\ell}
	 = \calB^{\ell+k}[\ipro[k]{0}{\vec i}:\ipro[k]{sm}{\vec i}]$, 
	so that Definition~\ref{d:cellsupp},
	\[
	\bbI\left(\ch[k]{\vphi_{\vec i}^\ell} \right)
	= \bbI\left[ \calB^{\ell+k}[\ipro[k]{0}{\vec i}:\ipro[k]{sm}{\vec i}] \right]	
	= \calI^{\ell+k}\left[\ipro[k]{0}{\vec i}: \ipro[k]{sm}{i}+p\right].
	\]
	On the other hand, due to Lemma~\ref{l:cell_hie},
	\[
	\ch[k]{\bbI(\vphi_{\vec i}^\ell)}
	=\ch[k]{\calI^\ell[\vec i: \vec i+p]} 
	= \calI^{\ell+k}\left[\ipro[k]{0}{\vec i}: \ipro[k]{s}{\vec i+p}\right],
	\]
	and~(i) follows from the fact that $\ipro[k]{sm}{\vec i} + p = 
\ipro[k]{s}{\vec i+p}$ as can be proved by induction on $k$, using that $n=s+1$ 
and $m=p+1$.
	
	In order to prove~(ii), observe first that, again from 
Lemma~\ref{l:bspline_hie} and Definition~\ref{d:cellsupp}
	\[
	\bbI\left(\ch[-k]{\vphi_{\vec i}^\ell} \right)
	= \bbI\left(\calB^{\ell-k}\left[\fdiv[k]{sp}{\vec i}:\fdiv[k]{0}{\vec 
i}\right]\right)
	= \calI^{\ell-k}\left[\fdiv[k]{sp}{\vec i}:\fdiv[k]{0}{\vec i}+p\right].
	\]
	Besides, 
	\[
	\ch[-k]{\bbI(\vphi_i^\ell)}
	= \ch[-k]{\calI^\ell[\vec i:\vec i+p]} 
	= \calI^{\ell-k}\left[\fdiv[k]{0}{\vec i}:\fdiv[k]{0}{\vec i + p}\right].
	\]
	Thus, (ii) follows from the fact that $ \fdiv[k]{0}{\vec i + p}\le 
\fdiv[k]{0}{\vec i} + p$.
\end{proof}

\begin{definition}[B-splines overlapping a cell]\label{d:splinesoverlap}
    Given $k\in\Z$ and $I\in\calI$ we define
    $\bbB^k(I) := \{\vphi \in \calB : I\in \bbI^{-k}(\vphi)\}$. If 
$\calJ\subset\calI$ then
    $\bbB^k(\calJ) := \cup_{I\in\calJ} \bbB^k(I)$.
\end{definition}

This last definition is reciprocal to the previous one, $\bbB^k(I) 
\subset \calB^{\ell_I+k}$ is the set of B-splines $\vphi$ of level $\ell_I+k$ 
whose support overlaps with $I$. We immediately obtain the following.

\begin{lemma}\label{l:p1}
	Let $I \in \calI$, $\vphi \in \calB$ and $k \in \Z$. Then $\vphi\in\bbB^k(I)$ if and only if $\bbI^k(\vphi)\cap\ch[k]{I}\not=\emptyset$.
\end{lemma}

Note that $\bbI$ and $\bbB$ map B-splines into cells and viceversa, so $\bbI^0$, 
$\bbB^0$ cannot be the identity operators, instead, $\bbI^0 = \bbI$ and 
$\bbB^0=\bbB$.

\begin{lemma}[Cells overlapping a box of B-splines]\label{l:cellbspline}
Let $k\in\Z^+$ then
\begin{enumerate}[(i)]
\item $\bbI^k[\calB^\l[\vec i: \vec j]]=\calI^{k+\l}[\ipro[k]{0}{\vec 
i}:\ipro[k]{s}{\vec j + p}]$
\item\label{l:cellbspline:2} $\bbI^{-k}[\calB^\l[\vec i: \vec j]]=
    \calI^{-k+\l}[\fdiv[k]{0}{\vec i}:\fdiv[k]{0}{\vec j + p}]$
\end{enumerate}
\end{lemma}
\begin{proof}
From Definitions~\ref{d:celloverlap} and~\ref{d:cellsupp} we have, for $\vec i 
\le \vec n \le \vec j$ and $k\in\Z$
\[
\bbI^k(\vphi_{\vec n}^\ell) = \ch[k]\bbI(\vphi_{\vec n}^\ell) = 
\ch[k]\calI^\ell[\vec n:\vec n+p].
\]
If $k > 0$, Lemma~\ref{l:cell_hie}~(i) yields
\[
\bbI^k(\vphi_{\vec n}^\ell) = \calI^{\ell+k}[\ipro[k]{0}{\vec 
n}:\ipro[k]{s}{\vec n+p}],
\]
and for $k < 0$, Lemma~\ref{l:cell_hie}~(ii) leads to
\[
\bbI^k(\vphi_{\vec n}^\ell) = \calI^{\ell+k}[\fdiv[k]{0}{\vec 
n}:\fdiv[k]{0}{\vec n+p}],
\]
and the assertions follow.
\end{proof}

\begin{lemma}[B-splines overlapping a box of cells]\label{l:bsplinecell}
Let $\l\in\Z^+_0$ and $\vec i, \vec j \in \Z^d$ then
\begin{enumerate}[(i)]
\item $\bbB^0[\calI^\l[\vec i: \vec j]]=\calB^{\l}[{\vec i} - p:{\vec j}]$
\item $\bbB^k[\calI^\l[\vec i: \vec j]]=
    \calB^{\l+k}[\ipro[k]{0}{\vec i} - p:\ipro[k]{s}{\vec j}]$,
for $k\in\Z^+$
\item\label{l:bsplinecell:3} $\bbB^{-k}[\calI^\l[\vec i: \vec j]]=
    \calB^{\l-k}[\fdiv[k]{0}{\vec i} - p:\fdiv[k]{0}{\vec j}]$,
for $k\in\Z^+$.
\end{enumerate}
\end{lemma}
\begin{proof}
Due to Definitions~\ref{d:celloverlap} and~\ref{d:cellsupp} we have that
$\vphi_{\vec n}^\ell \in \bbB^0(I_{\vec r}^\ell)$ iff
$I_{\vec r}^\ell \in \bbI(\vphi_{\vec n}^\ell)$, which holds 
iff $\vec n \le \vec r \le \vec n+p$ or $\vec r - p \le \vec n \le \vec r$.
This implies~(i).

From Definition~\ref{d:splinesoverlap} and Lemma~\ref{l:cellbspline}~(ii) we 
have, for $k \in \Z^+$ that
$\vphi_{\vec n}^{\ell+k} \in \bbB^k(I_{\vec r}^\ell)$ iff
$I_{\vec r}^\ell \in \bbI^{-k}(\vphi_{\vec n}^{\ell+k}) 
= \calI^\ell[\fdiv[k]{0}{\vec n}:\fdiv[k]{0}{\vec n+p}]$, 
which holds iff $\fdiv[k]{0}{\vec n}\le \vec r \le \fdiv[k]{0}{\vec n+p}$. 
Due to Lemma~\ref{l:inequalities} this is equivalent to 
$\vec n \le \ipro[k]{n-1}{\vec r}$ and $\ipro[k]{0}{\vec r} \le \vec n + p$.
The assertion~(ii) thus follows.

Analogously, for $k \in \Z^+$, 
$\vphi_{\vec n}^{\ell-k} \in \bbB^{-k}(I_{\vec r}^\ell)$ iff
$I_{\vec r}^\ell \in \bbI^{k}(\vphi_{\vec n}^{\ell-k}) 
= \calI^\ell[\ipro[k]{0}{\vec n}:\ipro[k]{s}{\vec n+p}]$, 
which holds iff $\ipro[k]{0}{\vec n}\le \vec r \le \ipro[k]{s}{\vec n+p}$. 
Due to Lemma~\ref{l:inequalities} this is equivalent to 
$\vec n \le \fdiv[k]{0}{\vec r}$ and $\fdiv[k]{0}{\vec r} \le \vec n + p$.
The assertion~(iii) thus follows.
\end{proof}

Next we state that basically, a B--spline overlaps a cell if and
only if it is positive on a sub-cell of it.

The next lemma is an immediate consequence of Definitions~\ref{d:cellsupp} and~\ref{d:splinesoverlap}.

\begin{lemma}\label{l:out_sup}
Let $k \in \Z^+$, $A^\l\subset\calI^\l$ and $\vphi\in\calB^{\l+k}$.
If $\vphi\not\in\bbB^k(A^\l)$ then $\vphi=0$ in the interior of $\cup_{I\in 
A^\l} I$.
\end{lemma}

\subsection{Overlapping chains}

In the current proofs of optimality for adaptive methods, and for some quasi-interpolants to provide local bounds, it seems necessary to have the \emph{level gap} of overlapping basis functions uniformly bounded. 
More precisely, whenever a cell is contained in the support of two basis functions, it is desirable that the difference in levels of those basis functions is uniformly bounded. 
This stems from the necessity of using inverse estimates in some stages of the proof. The difference could be large, but should be uniformly bounded. 
Some of the constants appearing in the results will depend on this bound, and the constants should be uniform to close the arguments.

That is why in this section we deal with B--splines overlapping other B--splines.

\begin{definition}[B--splines overlapping B--splines]\label{d:bsover1}
Let $\calH\subset\calB$, $\calF\subset\calB$ and $g\in\Z$ define
$\ovch{\calF}{g}{\calH} := \bbB^g(\bbI(\calF)) \cap \calH$. And
for simplicity we write $\ovch[]{\vphi}{g}{\calH}$ to denote 
$\ovch[]{\{\vphi\}}{g}{\calH}$ when $\vphi \in \calB$.
\end{definition}

\pedro{
\begin{remark}\label{R:ballofsplinessamelevel}
	Note that  $\psi \in \ovch[]{\vphi}{0}{\calB}$ if and only if $\ell_\psi = \ell_\vphi$ and $|\rho(\vphi,\psi)|< m$, so that $\ovch[]{\vphi}{0}{\calB} = B(\vphi,m-1,0)$, with $B(\cdot,\cdot,\cdot)$ as in Definition~\ref{d:ballofsplines}.
\end{remark}
}
\begin{definition}[Chains of overlapping B--splines]\label{d:bsover2}
Let $\calH\subset\calB$, $\calF\subset\calB$, \marce{$k \in \Z^+$} and $g\in\Z$ 
define $\ovch[k]{\calF}{g}{\calH}$ as the $k$-th 
fold composition of $\ovch{\cdot}{g}{\calH}$, i.e., 
$\ovch[k+1]{\calF}{g}{\calH} = \ovch{\ovch[k]{\calF}{g}{\calH}}{g}{\calH}$ and
$\ovch[1]{\calF}{g}{\calH} = \ovch{\calF}{g}{\calH}$.
\end{definition}

\pedro{
\begin{remark}\label{R:Osets}
It is worth noticing that the computational implementation of these concepts is very easy. It is just the intersection of sets of indices, which are previously grouped by levels.
\end{remark}
}

\begin{lemma}[Properties of overlapping chains]
\label{l:overlapping}
Let $g\in\Z^+$ and $\vec i, \vec j \in \Z^d$ and
$k=1,\dots,\lfloor\frac{\l} {g}\rfloor$ then 
\begin{enumerate}[(i)]
 \item $\ovch[k]{\calB^\l[\vec i:\vec j]}{-g}{\calB} =
  \calB^{\l - g k}(\operatorname{L}^{k}({\vec i}): \operatorname{R}^{k}({\vec 
j}))$, where 
  $\operatorname{L}$ and $\operatorname{R}$ are the index functions defined 
in~\eqref{eq:index_functions_L_R}.
\item\label{l:overlapping:rho} \marce{$\ovch[k]{\vphi}{-g}{\calB}
\subset \ball{\vphi}{C}{\ell_\vphi-gk}$, with $\ball{\cdot}{\cdot}{\cdot}$ the 
ball of B-splines from Definition~\ref{d:ballofsplines} and $C := p 
\left(\frac{1-1/n^{kg}}{n^g-1}\right)n^g + 1\leq \frac{p n^g}{n^g-1} + 1$.}
 \item \label{l:overlapping:branching}$\#\ovch[k]{\vphi^\l_{\vec i}}{g}{\calB} 
< 
 (2C+1)^d$.
\end{enumerate}
\end{lemma}
\begin{proof}
Using Definitions~\ref{d:bsover1}, 
Lemma~\ref{l:cellbspline}~\ref{l:cellbspline:2} and 
Lemma~\ref{l:bsplinecell}~\ref{l:bsplinecell:3}
we have
\begin{align*}
	\ovch[]{\calB^\ell[\vec i:\vec j]}{-g}{\calB}
	&= \bbB^{-g}(\bbI(\calB^\ell[\vec i:\vec j])) 
	 = \bbB^{-g}(\calI^\ell[\vec i: \vec j+p]) 
	\\
	&= \calB^{\ell-g}[\fdiv[g]{0}{\vec i}-p:\fdiv[g]{0}{\vec j+p}] 
	 = \calB^{\ell-g}(\operatorname{L}(\vec i):\operatorname{R}(\vec j)),
\end{align*}
due to~\eqref{eq:index_functions_L_R}. By induction~(i) follows.

In order to prove~(ii) observe that from~(i), $\psi \in \ovch[k]{\vphi_{\vec 
i}^\ell}{-g}{\calB}$ if and only if $\psi \in 
\calB^{\ell-gk}(\operatorname{L}^k(\vec i):\operatorname{R}^k(\vec i))$, which 
holds if and only if $\operatorname{L}^k(\vec i) \le \vec i_\psi \le 
\operatorname{R}^k(\vec i)$. Due to Lemma~\ref{lem:L_R_properties} this is 
equivalent to 
\[
\frac{\vec i}{n^{kg}} - \frac{p}{n^g-1}\frac{n^{kg}-1}{n^{(k-1)g}} - A
\le \vec i_\psi 
\le \frac{\vec i}{n^{kg}} + \frac{p}{n^g-1}\frac{n^{kg}-1}{n^{kg}} - B,
\]
for some $0 \le A,B\le 1 - \frac{1}{n^{kg}}$.
This, in turn, is equivalent to
\[
B - \frac{p}{n^g-1}\frac{n^{kg}-1}{n^{kg}}
\le \frac{\vec i}{n^{kg}} - \vec i_\psi 
\le \frac{p}{n^g-1}\frac{n^{kg}-1}{n^{(k-1)g}} + A,
\]
which implies
\[
|\rho(\vphi,\psi)| = 
\left| \frac{\vec i}{n^{kg}} - \vec i_\psi \right| 
\le p \frac{1-\frac{1}{n^{kg}}}{n^g-1}n^g+1,
\]
and~(ii) holds. The final assertion~(iii) is an immediate consequence of~(ii) 
and Lemma~\ref{l:radius}.
\end{proof}

\begin{remark}
	Notice that from Definition~\ref{d:bsover1}, $\eta \in 
\ovch[]{\vphi}{-k}{\calB}$ iff 
	$\eta \in \bbB^{-k}(\bbI(\vphi)) = \cup_{I \in \bbI(\vphi)} \bbB^{-k}(I)$,
	and due to Definition~\ref{d:splinesoverlap} this holds
	iff
	there exists $I \in \bbI(\vphi)$ with $\eta \in \bbB^{-k}(I)$, i.e., iff
	$\bbI(\vphi) \cap \bbI^k(\eta) \neq \emptyset$. Summarizing,
	\begin{equation}\label{ovch}
	\eta \in \ovch[]{\vphi}{-k}{\calB} \qquad\Leftrightarrow\qquad
	\bbI(\vphi) \cap \bbI^k(\eta) \neq \emptyset.
	\end{equation}
\end{remark}

\begin{definition}[Totally overlapped] 	
Let $\vphi\in\calB$ and $\calF\subset\calB$,
we say that $\vphi$ is \emph{totally overlapped} by \calF if there
is a partition $\mathcal{P}$ of $\bbI(\vphi)$ such that 
$\mathcal{P} \subset\bbI(\calF)$.
\end{definition}

\begin{lemma}[Overlapping of descendants]\label{l:over_ch}
Let $\vphi\in\calB$ and $\calH\subset\calB$
then  
\[
 \ovch{\ch[k]{\vphi}}{j}{\calH} \marce{=} \ovch{\vphi}{j+k}{\calH},
 \quad \text{for any $j\in\Z$ and $k \in \Z_0^+$.}
\]
\end{lemma}
%%%%%%%%%%%%%%%%
%%%

\begin{proof}
	Let $\vphi = \vphi_{\vec i}^\ell\in \calB$. From Definitions~\ref{d:bsover1} 
and~\ref{d:celloverlap} and Lemma~\ref{L:children}, 
	\[
	\ovch[]{\ch[k]{\vphi}}{j}{\calH} 
	= \bbB^j\left(\bbI(\ch[k]{\vphi})\right)\cap \calH
	= \bbB^j\left(\bbI^k(\vphi)\right)\cap \calH,
	\quad \text{for any $j\in\Z$ and $k \in \Z_0^+$.}
	\]
	Since by definition $\ovch{\vphi}{j+k}{\calH} = \bbB^{j+k}(\vphi)\cap\calH$, 
the rest of the proof will be devoted to proving that 
$\bbB^j\left(\bbI^k(\vphi)\right) = \bbB^{j+k}(\vphi)$ for any $j\in\Z$ and $k 
\in \Z_0^+$.
	
	Observe that Lemma~\ref{l:cellbspline} yields
	\[
	\bbB^j\left(\bbI^k(\vphi)\right)
	= \bbB^j\left(\bbI^k(\calB^\ell[\vec i:\vec i])\right)
	= \bbB^j\left(\calI^{k+\ell}\big[\ipro[k]{0}{\vec i}: \ipro[k]{s}{\vec 
i+p}\big]\right),
	\]
	for every $k \in \Z_0^+$ and $j \in \Z$.
	
	Consider first the case $j \in \Z_0^+$. From 
Lemma~\ref{l:bsplinecell}~(i)--(ii),
	\begin{align*}
		\bbB^j\left(\bbI^k(\vphi)\right)
		%&=\bbB^j\left(\calI^{k+\ell}[\ipro[k]{0}{\vec i}: \ipro[k]{s}{\vec 
%i+p}]\right)\\
		&= \calB^{j+k+\ell}\left(\ipro[j]{0}{\ipro[k]{0}{\vec i}}-p
		:\ipro[j]{s}{\ipro[k]{s}{\vec i + p}}\right)
		\\
		&= \calB^{j+k+\ell}\left(\ipro[j+k]{0}{\vec i}-p
		:\ipro[j+k]{s}{\vec i + p}\right)
		\\
		&= \bbB^{j+k}(\calI^\ell[\vec i:\vec i + p])
		= \bbB^{j+k}(\bbI(\vphi)).
	\end{align*}
	
	If $j \in \Z^-$, Lemma~\ref{l:bsplinecell}~(iii) yields
	\[
	\bbB^j\left(\bbI^k(\vphi)\right)
	= \calB^{k+\ell+j}\left(\fdiv[-j]{0}{\ipro[k]{0}{\vec i}}-p
	   : \fdiv[-j]{0}{\ipro[k]{s}{\vec i + p}} \right).
	\]
	
	Consider now $j < 0$ fixed and $k \ge -j$ ($j+k\ge 0$), then
	$\fdiv[-j]{0}{\ipro[k]{0}{\vec i}} = \ipro[k+j]{0}{\vec i}$ and 
	$\fdiv[-j]{0}{\ipro[k]{0}{\vec i + p}} = \ipro[k+j]{0}{\vec i + p}$,
	so that Lemma~\ref{l:bsplinecell}~(i)--(ii) leads to
	\[
	\bbB^j\left(\bbI^k(\vphi)\right)
	= \calB^{k+\ell+j}\left[ \ipro[k+j]{0}{\vec i}-p:\ipro[k+j]{0}{\vec i + p} 
\right]
	= \bbB^{j+k}(\bbI(\vphi)).
	\]
	If $k < -j$ ($k+j < 0$), then
	$\fdiv[-j]{0}{\ipro[k]{0}{\vec i}} = \fdiv[-(k+j)]{0}{\vec i}$ and 
    $\fdiv[-j]{0}{\ipro[k]{0}{\vec i + p}} = \fdiv[-(k+j)]{0}{\vec i + p}$,
    so that,
	\[
	\bbB^j\left(\bbI^k(\vphi)\right)
	= \calB^{k+\ell+j}\left( \fdiv[-(k+j)]{0}{\vec i}-p:\fdiv[-(k+j)]{0}{\vec i + 
p} \right)
	= \bbB^{j+k}(\bbI(\vphi)),
	\]
	due to Lemma~\ref{l:bsplinecell}~(iii).
	
	Summarizing, for each $j \in \Z$ and any $k \in \Z_0^+$, 
	$ \bbB^j\left(\bbI^k(\vphi)\right) = \bbB^{j+k}(\bbI(\vphi))$,
	and the assertion thus follows.
\end{proof}

\section{Hierarchical Generators and Spaces}\label{s:generators}
%%%%%%%%%%%%%%%%%%%%%%%%%%%%%%%%%%
\pedro{We are interested in a class of spaces obtained through an iterative process of function refinement.
These spaces are similar to other spaces that have been previously defined in the literature; see Remark~\ref{r:coincidencia}.
Our approach is based on functions rather than on subdomains, and yields a particular class of subsets of $\calB$ that we call \emph{lineages} and are given by the following definition.}
%%%%%
\begin{definition}[Lineage set]\label{d:lineage}
A set $\calL \subset \calB$, is called a \emph{lineage} if it is finite and 
$\calL\subset \calB^0\cup\ch{\calL}$. 
Given a lineage set \calL we let 
$\calC := \calB^0\cup\ch{\calL}$, be the children of \calL plus the coarsest 
B-splines, that we will call the $\calC$-set associated 
to the lineage.
\end{definition}

This definition is very simple, resorting to the operator $\ch[]{}$ and 
notation from set theory. It says, essentially, that a set $\calL$ is a lineage 
if every element of $\calL$ is the child of an element of $\calL$ or is 
itself an element of level zero (belongs to $\calB^0$). The well known 
\emph{tree} structure fulfills this assumption, among others.
\pedro{This new framework allows us to deal with a simple implementation, which also makes it very easy to control the overlapping of functions from different levels.}

\pedro{The idea behind a lineage $\calL$ is that $\calL$ is the set of functions that have been \emph{refined} in an adaptive process, so that the \emph{hierarchical space} is the one spanned by their children. More precisely. 
	%The \emph{hierarchical generator} is the analogous, in our context, to the leaves of a tree, and is defined as follows.
}
\begin{definition}[Hierarchical generator]\label{d:generator}
Let \calL be a lineage, the set 
\[
\calH = \big(\calB^0\cup\ch{\calL} \big) \setminus \calL = \calC \setminus \calL
\]
is the \emph{hierarchical generator} corresponding to \calL.
%The $\calC$-set associated to $\calL$ is $\calC = \calB^0\cup\ch{\calL}$.
\end{definition}

Notice that $\calL = \emptyset$ is a valid lineage, and its corresponding generator is $\calH = \calB^0$.
It is convenient to have a notation to arrange these sets
by level, %generations 
so for $\l\in\Z^+_0$ let
$\calC^\l := \calC\cap\calB^{\l}$,
$\calL^\l := \calL\cap\calB^{\l}$ and 
$\calH^\l = \calH \cap \calB^\l$.
It is easy to see that for any lineage, $\calC^0 = \calB^0$, 
$\calC=\dcup_{\l=0}^{\infty}\calC^\l$ and there exists $\ell $ such that 
$\calL^\ell = \emptyset$.
If $\calL^\l=\emptyset$ then $\calC^{\l'}=\emptyset$ for all $\l'>\l$, 
and the following is well defined.
\begin{definition}[Depth of a lineage]
Given the hierarchical generator \calH with lineage \calL 
we define its depth as
$\depth(\calL) = \depth(\calH):=\min\{\l:\calL^\l=\emptyset\}$.
\end{definition}
Observe that $\calL$ has functions of level $\depth(\calL)-1$ and $\calH$ 
has functions of level $\depth(\calH)$, which is the finest level of functions in $\calH$.
%%%%%%%%%%%%%%%%%%%%%%%%%%%%%%%%%

From the definition of hierarchical generator it is clear that for each
lineage there is a unique hierarchical generator.
The reciprocal is also true.
%%%%%%%%%%%%%%%%%%%%%%%%%%%
\begin{lemma}[Lineage to generator bijection]\label{l:bijection}
    There is a bijection between hierarchical generators and lineages.
\end{lemma}
%%%%%%%%%%%%%%%%%%%%%%%%%%%
\begin{proof}
    Let  $\calL$ and $\bar\calL$ be two lineages giving the same hierarchical 
    generator \calH.
    That is to say $\calC\setminus\calL = \bar\calC\setminus\bar\calL$, or
    by levels using the symmetric difference 
    $(\calL^\l \triangle \calC^\l) \triangle (\bar\calC^\l\triangle\bar\calL^\l)
    =\emptyset$ for each \l.
    For $\l=0$, $\calC^0=\bar\calC^0=\calB^0$ so using the symmetric difference
    property that $(A \triangle B) \triangle (B \triangle C) = (A \triangle C)$
    we get that $\calL^0=\bar\calL^0$.
    Assume that $\calL^\l=\bar\calL^\l$ for $\l\leq n$.
    Then $\calC^{n+1}=\bar\calC^{n+1}$ and so using the previous 
    argument on the symmetric difference we obtain that 
$\calL^{n+1}=\bar\calL^{n+1}$.
    Hence, we have shown by induction that $\calL = \bar\calL$ 
    (and that $\calC = \bar\calC$), thus proving that
    there is a unique lineage associated to each hierarchical generator.
\end{proof}
%%%%%%%%%%%%%%%%%%%%%%%%%%%

\pedro{
\begin{remark}[Lineages vs.\ $\calC$-sets]	
    This is subtle but important. One may be tempted to use the $\calC$-sets to identify the hierarchical generators instead of the lineages.
    However, the relation between hierarchical generators and the $\calC$-sets is not one-to-one, as is the case between hierarchical generators and lineages.
%    This is subtle but important, we have shown that the correspondence between 
%hierarchical generators and lineages is one to one.
%    However, this is not the case for the relation between hierarchical generators and the corresponding $\calC$-sets, which one may be tempted to use as .
    In fact, consider $m=3$ with $d=1$, let 
    $\calL=\{\vphi^{0}_{-2}, \vphi^{0}_{0}\}$ and 
    $\bar\calL=\{\vphi^{0}_{-2},\vphi^{0}_{-1},\vphi^{0}_{0}\}$ then the 
    corresponding $\calC$-awta are $\calC = \bar\calC = \calB^0 \cap \calB^1$, even though the corresponding hierarchical generators differ, i.e., 
    the same set \calC can correspond to different hierarchical generators.
    Thus a lineage has some built-in information that is missing in the $\calC$-sets.
    Something similar happens with the so called
    hierarchical grids \cite{VGJS2011, BG2016} given by nested
    domains, where every grid leads to a generator but different grids
    may lead to the same generator.
\end{remark}
}

\begin{definition}[Hierarchical Space]
Given a hierarchical generator \calH, the linear space 
$\bbV=\spn\calH$ is called a hierarchical space.
\end{definition}

\begin{lemma}[More relations between \calH and \calL]\label{l:H_L_props}
Given a hierarchical generator \calH with the associated
lineage \calL  it follows that
 \begin{enumerate}[(i)]
 \item\label{l:H_L_props:L_from_H} 
 Each $\psi\in\calL$ can be written as a linear combination of its 
\emph{descendants} in~$\calH$. More precisely, 
  $\psi \in \spn(\calH\, \cap\, \descs{\psi})$, where $\descs{\psi} = 
\cup_{k\in\Z^+} \ch[k]{\psi}$ is the set of all \emph{descendants} of $\psi$.
\item\label{l:H_L_props:H_desc_of_L} Each $\vphi\in\calH$ has an ancestor of 
every possible level in
\calL. More precisely, for $\ell \in \Z_0^+$, $k \in \Z^+$ it holds
that
$\calH^{\l+k}\subset\ch[k]{\calL^\l}$. 
 \end{enumerate}
\end{lemma}
\begin{proof}
  We thus prove (i) by 
(backward) induction on the level of $\psi$. Let $N = \depth \calL$ and $\psi 
\in \calL^{N-1}$, i.e., $\psi \in \calL$, with $\l_\psi=N-1$, then no child of 
$\psi$ belongs to $\calL$, because $\calL^N = \emptyset$. Therefore, 
$\ch{\psi}\subset\calC\setminus\calL = \calH$ so $\psi \in \spn(\calH \cap 
\descs{\psi})$ due to Lemma~\ref{l:cp}.
  Suppose now that the assertion is true for all functions in $\calL^{N-j}$ 
with 
$0\le j < N$. Let $\psi \in \calL^{N-(j+1)}$. Lemma~\ref{l:cp} yields $\psi \in 
\spn{\ch \psi}$.
  Since $\calL$ is a lineage, Definition~\ref{d:lineage} implies that 
$\ch[]{\psi} \subset \calC$. Then, from Definition~\ref{d:generator}, each 
child 
of $\psi$ either belongs to $\calH$ or to $\calL^{N-j}$. Each of the latter 
belongs to $\spn(\calH\cap\descs{\psi})$ from the inductive assumption, and the 
assertion follows.
\marce{In order to prove (ii), we proceed by induction on $k$. Let 
$\ell\in\Z_0^+$, $k=1$ and $\psi \in \calH^{\ell+1}$, i.e. $\psi\in\calH$ and 
$\ell_\psi=\ell+1\geq 1$. Then $\psi\notin\calB^0$, thus $\psi \in \ch{\calL}$, 
i.e. there exist $\phi\in\calL$ with $\ell_\phi=\ell_\psi-1=\ell$ such that 
$\psi\in\ch{\phi}\subset\ch{\calL^\ell}$. Suppose now that the assertion is 
true 
for $k=m$ and $\ell \in\Z_0^+$. Let 
$\psi\in\calH^{\ell+m+1}=\calH^{(\ell+1)+m}$, then from the inductive 
assumption 
$\psi\in\ch[m]{\calL^{\ell+1}}\subset 
\ch[m]{\ch{\calL^\ell}}=\ch[m+1]{\calL^\ell}$, so the assertion follows.}
\end{proof}

\begin{remark}\label{R:spanCequalsspanH} It is worth noticing that as an immediate consequence of the previous lemma, we always have
$\spn\calL\subset\spn\calC = \spn\calH$.
\end{remark}

\begin{corollary}[\calH and \calL cell relations]\label{c:H_L_props}
For any hierarchical generator \calH and 
$\ell \in \Z_0^+$, $k \in \Z^+$, we have
\begin{enumerate}[(i)]
\item\label{c:H_L_props:L_overlap_by_dsc} If $\psi\in\calL$ then
$\bbI(\psi)\subset \cup_{k>0}\bbI^{-k}(\ch[k]{\psi}\cap\calH)$,
thus $\bbI(\calL^\l) \subset \cup_{k>0}\bbI^{-k}(\calH^{\l+k})$.
\item \label{c:H_L_props:H_overlap_by_L}$\bbI(\calH^{\l+k}) \subset 
\bbI^k(\calL^\l)$.
\end{enumerate}
\end{corollary}
\begin{proof}
To show part \ref{c:H_L_props:L_overlap_by_dsc}, let
$\psi\in\calL$ and consider a given $I\in\bbI(\psi)$.
Then for any $x\in I^o$ as $\psi(x)\not=0$ Lemma
\ref{l:H_L_props}\ref{l:H_L_props:L_from_H} implies
that there is
$\vphi\in\ch[k]{\psi}$ such that $\vphi(x)\not=0$ thus there
is $I'\in\bbI(\vphi)$ such that $\ch[-k]{I'} = I$.
From here part \ref{c:H_L_props:L_overlap_by_dsc} follows.
In order to prove~(ii) we use  
Lemma \ref{l:H_L_props}\ref{l:H_L_props:H_desc_of_L}
to see that
$\bbI(\calH^{\l+k})\subset\bbI(\ch[k]{\calL^\l})$ and Lemma~\ref{L:children} to 
conclude that $\bbI(\ch[k]{\calL^\l}) = \bbI^k(\calL^\l)$.
\end{proof}

\begin{comment}
\begin{lemma}[Total overlap by finer functions in \calH]%
	\label{l:lin_totally_overlap_by_finer_H}
Let $\calL$ be a lineage. For any $\vphi\in\calL$ we have
$\bbI(\vphi) \subset \cup_{k>0}\bbI^{-k}(\ovch{\vphi}{k}{\calH})$.
% and $I\in\bbI(\vphi)$ then there
% exists $\eta\in\calH^{\l+k}$  for some $k\in\Z^+$ such that 
% $I\in\bbI^{-k}(\eta)$, i.e. $\vphi\overlap\eta$.\todo{$\overlap$ no esta def}
\end{lemma}

\begin{proof} 
	\sebatodo[inline]{Probablemente hay que usar el Lemma~\ref{l:H_L_props}}
\end{proof}
\end{comment}

%\pedro{We now correlate our spaces with previous definitions of hierarchical spaces, given by hierarchical \emph{grids}}
\pedro{Other definitions of hierarchical spline spaces are given in terms of hierarchical \emph{grids}, or sequence of nested subdomains. In those definitions it is natural to think of \emph{active cells}, which we now define.}

\begin{definition}[Active cells]\label{d:activecells}
Given a hierarchical generator $\calH$ we define the set of \emph{active cells} 
as 
$\calA := \bbI(\calC) \setminus \bbI(\calL)$.
And $\calA^\l = \calA \cap \calI^\ell$, for $\ell \in \Z_0^+$.
\end{definition}

Observe that $\calA\subset\bbI(\calH)$ but not the other way around,
The next result shows that finer B--splines in the generator
do not overlap active cells.
\begin{lemma}\label{l:H_L_props:H_overlap_by_L1}
Given a hierarchical generator $\calH$ we have
$\calH^{\l+k} \cap \bbB^k(\calA^\l)=\emptyset$ for $k \in \Z^+$ and $\ell \in 
\Z_0^+$.
In other words, 
if $\vphi\in\calH^{\l+k}$ and $I\in \calA^\l$ then
$\vphi=0$ on \seba{$I$}.
\end{lemma}
\begin{proof}
Let $\ell \in \Z_0^+$, $k \in \Z^+$ and $\vphi\in\calH^{\l+k}$, 
then by Corollary~\ref{c:H_L_props}\ref{c:H_L_props:H_overlap_by_L},
$\bbI(\vphi)\subset\bbI^k (\calL^\l)$.
From Definition~\ref{d:activecells} we have 
%, on the one hand,
${\calA^\l}\cap \bbI(\calL^\l) = \emptyset$, 
and also $\ch[k]{\calA^\l}\cap \bbI^k(\calL^\l) = \emptyset$.
Therefore, $\bbI(\vphi) \cap \ch[k]{\calA^\l} = \emptyset$, 
hence $\vphi \notin \bbB^k(\calA^\l)$ owing to Lemma~\ref{l:p1}, and the 
assertion follows.
\end{proof}

\begin{lemma}[Positive spanning of the unity]\label{l:pu}
Let $\calH$ be a hierarchical generator. Then 
for each $\vphi\in\calH$
there exists a positive coefficient $c_\vphi$ such that
$\sum_{\vphi\in\calH} c_\vphi \vphi = 1$.
\end{lemma}
\begin{proof} 
As $\calB^0$ is a partition of unity over $\Omega$, it follows that
$1=\sum_{\vphi\in\calB^0}\vphi$. Also we know that
$\calB^0 = \calH^0 \dbigcup \calL^0$ so
$1=\sum_{\vphi\in\calH^0}\vphi + \sum_{\vphi\in\calL^0} \vphi$.
Suppose we have shown that 
$1=\sum_{j=0}^{\l}\sum_{\vphi\in\calH^j}  c_\vphi \vphi + 
\sum_{\vphi\in\calL^\l}\beta_\vphi \vphi$
with $c_\vphi>0$ for each $\vphi\in\bigcup_{j=0}^{\l} \calH^j$
and $\beta_\vphi>0$ for each $\vphi\in\calL^{\l}$.
Using Lemma~\ref{l:cp}~\ref{l:cp:ch_coef}, each $\vphi \in \calL^\ell$ can be 
spanned by its children in $\calC^{\l+1} = \calH^{\l+1} \dcup \calL^{\l+1}$ 
with 
positive coefficients, so that
$\sum_{\vphi\in\calL^\l}\beta_\vphi \vphi
= \sum_{\vphi\in\calH^{\l+1}}  c_\vphi \vphi + 
\sum_{\vphi\in\calL^{\l+1}}\beta_\vphi \vphi $, with $c_\vphi,\beta_\vphi > 0$, 
and thus
$1=\sum_{j=0}^{\l+1}\sum_{\vphi\in\calH^j}  c_\vphi \vphi + 
\sum_{\vphi\in\calL^{\l+1}}\beta_\vphi \vphi$.
\end{proof}

\section{Refinement}\label{s:refinement}
Lineages provide a convenient framework to define a 
concept of refinement that will allow us to rigorously study
the process. The germ is the following definition.
\begin{definition}[Refinements and refiner sets]\label{d:refinement}
 We say that a lineage $\calL_*$ is a \emph{refinement} of the lineage $\calL$ 
whenever  $\calL \subset \calL_*$, and we denote it with $\calL_*\refi\calL$.
 The set difference $\calR = \calL_*\setminus\calL$ is called the \emph{refiner 
set}
 of the refinement.
 Accordingly (in light of Lemma \ref{l:bijection}) we say that a hierarchical 
generator $\calH_*$ is a refinement of $\calH$, and denote it with 
$\calH_*\refi\calH$, whenever $\calL_*\refi\calL$.
\end{definition}

\begin{remark}[Conventional notation]\label{r:convention}
	From now on, unless explicitly stated, 
	whenever we say that $\calL$, $\calL_*$ are lineages, without further stating,
	$\calH$, $\bbV$, $\calH_*$, $\bbV_*$ will denote their corresponding 
hierarchical generators and spaces, respectively, and vice versa. Moreover, if 
$\calL_* \refi \calL$, $\calR = \calL_* \setminus \calL$ will be the refiner 
set.
\end{remark}

As an immediate consequence of Lemma~\ref{l:H_L_props} and the fact that 
$\calL_* \refi \calL$ yields $\calC_* \supset \calC$  
we have that 
$\bbV_*\supset\bbV$ due to Remark~\ref{R:spanCequalsspanH}.
But notice that $\calH_*\refi\calH$ does not necessarily imply that 
$\calH_*\supset\calH$.

\subsection{Order on Refinements and the Smallest Element}
The set of all lineages with the inclusion relation is a partially ordered set 
(POSET). 
A \emph{minimal element} of a subset $S$ of some POSET is defined as an element 
of $S$ that is not greater than any other element in $S$.
The \emph{least element} is an element of $S$ that is smaller than every other 
element of $S$.
A set can have several minimal elements without having a least element. 
However, 
if it has a least element, it can't have any other minimal element.
As the family of lineages is a POSET and there is a one to one 
correspondence with the family of hierarchical generators ($\calH_* \refi 
\calH$ 
iff $\calL_* \supset \calL$) we transfer the partial order from the lineages to 
the hierarchical generators.
More precisely.
\begin{property}[Generators are partially ordered by refinement]
 The ``being refinement of'' relation $\refi$ is a partial order
 in the family of hierarchical generators.
\end{property}

The family of hierarchical generators has a least element $\calH=B^0$, which 
corresponds to $\calL=\emptyset$.

The approach to refinement as a partial order allows us to rigorously pose the 
problem of finding the smallest (minimal or least) refinement of 
\calH that satisfies some given property.
For example, if given a hierarchical generator \calH we define 
$\textgoth{B}(\calH):=\{\calH_*\refi\calH : \calH_* \text{ is linearly 
independent} \}$,
we can ask what is $\min \textgoth{B}(\calH)$, the set of minimal elements;
in some cases of interest it can be a singleton with only the least element.

\subsection{Algebra of a refinement}

The algebra of set inclusion can be applied to
Definition \ref{d:refinement} and obtain some useful properties with simple 
proofs.
\begin{lemma}[Basic properties]\label{l:ref_identities}
 Let $\calH_* \refi \calH$ with $\calL_*$, $\calL$ the
corresponding lineages and \calR the refiner set,
and let $\calM:=\calR\cap\calH$.
 Then 
 \begin{enumerate}[(i)]
  \item \label{l:ref_identities:HrminusH}
  $\calH_*\setminus\calH=\ch{\calR}\setminus(\calC\cup\calR)$

  \item \label{l:ref_identities:HminusHr}
  $\calH_*\cap\calH = \calH\setminus\calR$

  \item \label{l:ref_identities:HminusHstar}
  $\calH\setminus\calH_* = \calM$

\item \label{l:ref_identities:RminusH}
$\calR\setminus\calH = \calR\setminus\calC$

\item \label{l:ref_identities:chR}$\calR \setminus \calH \subset 
\ch{\calR}\setminus\calC$

\item \label{l:ref_identities:R}
$\calR \subset (\ch{\calR}\setminus \calC) \dcup \calM$
%$\calR =(\ch{\calR}\setminus (\calH_*\cap\calH)) \dcup \calM$

%\item \label{l:ref_identities:chR1}
%$\ch{\calR}\setminus\calC=(\calH_*\setminus\calH) \dcup (\calR \setminus 
% \calC)$
%\item \label{l:ref_identities:R0}$\calR^0 \subset\calH$

 \end{enumerate}
\end{lemma}
\begin{proof}
Using that $\calH = \calC\setminus\calL$, $\calH_* = \calC_*\setminus\calL_*$, 
and the set identity 
\begin{equation}\label{e:setminus}
(A\setminus B)\setminus(C\setminus D) = (A\setminus (B\cup C))\cup((A\cap 
D)\setminus B)
\end{equation}
we get
\begin{equation*}
 \calH_*\setminus\calH = (\calC_*\setminus\calL_*) \setminus 
(\calC\setminus\calL) 
                       =(\calC_*\setminus (\calL_*\cup \calC))\cup
                        \underbrace{((\calC_*\cap \calL)\setminus 
\calL_*)}_{\subset \calL \setminus \calL_* =\emptyset}
                        =\calC_*\setminus(\calC\cup\calR),
\end{equation*}
where in the last equality we have used that $\calL_*\cup \calC = \calL \cup 
\calR \cup \calC = \calC\cup\calR$, because $\calL \subset \calC$. Finally,  
$\calC_*\setminus(\calC\cup\calR)= 
(\ch{\calL_*}\cup\calB_0)\setminus(\ch{\calL}\cup\calB_0\cup\calR)= 
(\ch{\calL_*}\setminus\ch{\calL})\setminus(\calC\cup\calR)$ 
and~\ref{l:ref_identities:HrminusH} follows.
%%%

Identity~\ref{l:ref_identities:HminusHr} follows from the set identity 
$(A \setminus B) \cap (C \setminus D) = (A\cap C) \setminus (B\cup D)$ 
and the fact that $\calC\subset\calC_*$ and $\calL\subset\calL_*$. Indeed,
\begin{equation}
\begin{aligned}
 \calH_*\cap\calH &= (\calC_*\setminus\calL_*) \cap (\calC\setminus\calL)
                   = (\calC_*\cap\calC) \setminus (\calL_*\cup\calL)\\
                  &=\calC\setminus\calL_* = \calC\setminus(\calL\cup\calR)
                   =(\calC\setminus\calL)\setminus\calR = \calH\setminus\calR.
\end{aligned}
\end{equation}

For~\ref{l:ref_identities:HminusHstar} observe that $\calH\setminus\calH_* 
=\calH\setminus(\calH_*\cap\calH)$
so from~\ref{l:ref_identities:HminusHr},
\[
\calH\setminus\calH_* = \calH\setminus (\calH\setminus\calR) = 
\calH\cap\calR = \calM.
\]
Using the set identity $ A\setminus(B\setminus C) = (A\setminus B)\cup(A\cap 
C)$,
we conclude
\[
\calR\setminus\calH = \calR\setminus(\calC\setminus\calL) = 
(\calR\setminus \calC)\cup(\calR\cap \calL) = \calR\setminus\calC,
\]
because $\calL \cap \calR = \emptyset$ and~\ref{l:ref_identities:RminusH} 
follows.

To prove \ref{l:ref_identities:chR} we use~\ref{l:ref_identities:RminusH} and 
the fact that $\calL_*$ is a lineage, to conclude that
\begin{align*}
\calR \setminus \calH 
&=\calR \setminus \calC 
\subset \calL_*\setminus \calC  
\subset (\ch[]{\calL_*} \cup \calC)\setminus \calC 
= \ch[]{\calL_*} \setminus \calC 
\\
&\subset \ch[]{\calL_*} \setminus (\calC\cup\ch[]{\calL})
\subset (\ch[]{\calL_*} \setminus \ch[]{\calL} ) \setminus \calC
\subset \ch[]{\calR} \setminus \calC.
\end{align*}

In order to prove~\ref{l:ref_identities:R} observe that $\calR = (\calR 
\setminus \calH) \dcup (\calR \cap \calM) = (\calR\setminus\calH) \dcup \calM$ 
and use~\ref{l:ref_identities:chR}.
\end{proof}

\subsection{Single refinement}

Here we show that Definition \ref{d:refinement} of refinement is
in fact equivalent to the natural one of refining one function at a time.

Refining a function in a hierarchical generator
means to find the smallest refinement for which that function
is in the refiner set. More precisely, 
\begin{definition}[Refinement of one function]\label{d:single_ref}
Let \calH be a hierarchical generator and $\vphi\in\calH$,
the refinement of $\vphi$ is the least element of the 
family
$\{\bar\calH : \bar\calH\refi\calH \text{ and } \vphi\in\bar\calL \}$.
%$\{\calH_* : \calH_*\refi\calH \text{ and } \vphi\in\calL_* \}$.
\end{definition}

The next Lemma shows that the definition is well posed.
\begin{lemma}[Good definition]\label{l:good_def_sr}
Let \calH be a hierarchical generator and $\vphi\in\calH$,
the least element of 
$\{\bar\calH : \bar\calH\refi\calH \text{ and } \vphi\in\bar\calL \}$
is the hierarchical generator whose lineage is 
${\calL_*} = {\calL} \cup \{\vphi\}$, namely
$\calH_* = \calH\setminus\{\vphi\} \dcup (\ch{\vphi}\setminus\ch{\calL})$.   
\end{lemma}
\begin{proof}
    All we have to prove here is that if $\calL$ is a lineage and $\vphi$ is an 
element of the corresponding hierarchical generator $\calH$, then ${\calL_*} = 
{\calL} \cup \{\vphi\}$ is also a lineage. This is very simple, since 
$\vphi \in \calH = \calC \setminus \calL\subset \calC$ and $\calL$ is a lineage 
($\calL \subset \calC$) we have
    \[
    \calL_* = \calL \cup \{\vphi\} \subset \calC\cup \{\vphi\} = \calC 
     = \calB^0 \cup \ch[]{\calL} \subset \calB^0 \cup \ch[]{\calL_*} = \calC_*.
    \]
    The minimality is a consequence of the fact that the smallest set that 
contains $\calL$ and $\vphi$ is $\calL \cup \{\vphi\}$.
To find expressions for $\calH_*$ we use Lemma~%
\ref{l:ref_identities}
\ref{l:ref_identities:HrminusH}--\ref{l:ref_identities:HminusHr}
as follows
\begin{align*}
\calH_* &=  (\calH_* \cap \calH) \dcup (\calH_* \setminus \calH)
=  (\calH \setminus \calR) \dcup \big(\ch[]{\calR} \setminus (\calC \cup 
\calR)\big)
\\
&=  (\calH \setminus \{\vphi\}) \dcup (\ch[]\vphi \setminus \calC ) 
= (\calH \setminus \{\vphi\}) \dcup (\ch[]\vphi \setminus \ch\calL ).
\end{align*}
\end{proof}

\begin{remark}[Single refinement $=$ adding more resolution]
According to Lemma \ref{l:good_def_sr},
refining a B--spline function $\vphi$ consists in substituting it
by its children that are not already in the generator.
%%In the light of Lemma \ref{l:ml_span}\ref{l:ml_span:ch_optim}, 
We have added the smallest number of children that are
necessary to span $\vphi$.
% the smallest set of B--spline functions strictly finer that are
% necessary to span $\vphi$. 
\end{remark}

A constructive process to build $\calL_*$ is called \Call{SingleRefine}{}
and it is described in Algorithm \ref{a:single_ref} below, which we include 
despite its simplicity as it will be called from more complex algorithms later 
on.
%%%%%%%%%%
\begin{algorithm}[!h]
\caption{Refine one function}\label{a:single_ref}
\begin{algorithmic}[1]
\Function{SingleRefine}{$\calH$,$\vphi$}
\Require $\calH$ a hierarchical generator and $\vphi\in\calH$
\State ${\calL} \gets {\calL} \bigcup \{\vphi\}$
\EndFunction
\end{algorithmic}
\end{algorithm}

\begin{lemma}[Single refinement]\label{l:single_ref}
Let $\calH$ be a hierarchical generator with lineage $\calL$ and
$\vphi\in\calH$.
A call to \Call{SingleRefine}{$\calH$,$\vphi$}
modifies the set \calH, as \calL is modified [see Remark \ref{r:convention}].
Let the original set before the call be $\calH$, and $\calH_*$ its modification
after Algorithm \ref{a:single_ref} is executed.
Then $\calH_*$ is a refinement of the original $\calH$ with refiner 
$\calR=\{\vphi\}$.
Furthermore, $\calH_* = \calH \setminus \{\vphi\} \cup (\ch{\vphi}\setminus 
\calC^{\l_\vphi+1})$

\end{lemma} 
%%%
\begin{remark}[Implementation tip]
The modification for \calH is to remove $\vphi$ and
add $\ch{\vphi}$ that are not in $\ch[]{\calL}$. But the only children of 
$\vphi$ that could belong to $\ch[]{\calL}$ are those who are in 
$\ch{(\ch[-1]{(\ch{\vphi})}\cap\calL)}$. %\cap\calL$. 
This set looks more complicated than $\ch[]{\calL}$ but is in fact much smaller.
\end{remark}

\subsection{Refining a set of functions}
The process of refining one function can be naturally extended to
refining a subset \calM of \calH.
This can be defined as finding the smallest refinement of \calH whose
refiner contains \calM.
\begin{algorithm}[h]
\caption{Refine a set of functions}\label{a:refine0}
\begin{algorithmic}[1]
\Function{Refine}{$\calH$,$\calM$}
\Require $\calM \subset \calH$
\For{$\vphi \in \calM$}
\State \Call{SingleRefine}{$\calH$,$\vphi$}
\EndFor
\EndFunction
\end{algorithmic}
\end{algorithm}
%%%%
\begin{lemma}[Refining a subset of \calH]\label{l:refine0}
Let $\calH$ be a hierarchical generator and $\calM\subset\calH$.
A call to \Call{Refine}{$\calH$, $\calM$} of Algorithm \ref{a:refine0} 
finishes modifying the set \calH.
Let the original set before the call be $\calH$, and $\calH_*$ its 
modification after Algorithm \ref{a:refine0} is executed.
Then $\calH_*$ is a refinement of $\calH$ and
its refiner is $\calR=\calM$ ($\calL_* = \calL \dcup \calM$).
In particular, the same hierarchical generator $\calH_*$ is
obtained independent of the order in which the functions
of $\calM$ are passed to \Call{SingleRefine}{}.
Furthermore this is the smallest refinement of \calH that refines
all functions in \calM.
\end{lemma}
\begin{proof}
 To show that Algorithm \ref{a:refine0} finishes successfully
 we must ensure that the precondition of Algorithm \ref{a:single_ref}
 is satisfied.
 Let us order the elements of \calM in a sequence $(\vphi_0,\dots,\vphi_N)$
 and call \Call{SingleRefine}{$\calH$,$\vphi_i$} following that order.
 Let $\calH_0=\calH$ before the first call, 
 and $\calH_{i+1}$ the state of \calH after the $i$-th call.
 Now we proceed by induction.
 Clearly, $\calH_0$ is a hierarchical generator, $\calL_0=\calL$
 and $\{\vphi_0,\dots,\vphi_N\}\subset\calH_0$.
 Assume now that $\calH_i$ is a hierarchical generator, 
 $\calL_i=\calL\dcup\{\vphi_0,\dots,\vphi_{i-1}\}$
 and $\{\vphi_i,\dots,\vphi_N\}\subset\calH_i$.
 Under these conditions Lemma \ref{l:single_ref} states that 
 $\calH_{i+1}$ is a hierarchical generator, 
 $\calL_{i+1}=\calL_i\dcup\{\vphi_{i}\}=\calL\dcup\{\vphi_0,\dots,\vphi_{i}\}$,
 and that $\{\vphi_{i+1},\dots,\vphi_N\}\subset\calH_{i+1}$.

Thus we have shown by induction that for any order of the function in $\calM$
the sequential execution of \Call{SingleRefine}{} will finish giving
a hierarchical generator with lineage $\calL_*=\calL\cup\calM$.
Now this lineage is the same independent of order given to the functions of
\calM, so by Lemma \ref{l:bijection} they all give the same and unique
hierarchical generator $\calH_*$.

That this is the smallest refinement follows trivially from the fact
that $\calL\cup\calM$ is a lineage.
\end{proof}

%\subsection{Characterization of \calR by single refinements}
%\pedrotodo{empezar hablando de \calR?}
The most trivial process to construct a hierarchical generator is 
by a sequence of single refinements starting from $\calB^0$.
More precisely, let $N$ be a natural number, $\calH_0 = \calB^0$,
and for $i\in[0:N-1]$ let $\vphi_i\in\calH_i$ and 
$\calH_{i+1}$ the output of \Call{SingleRefine}{$\calH_i,\vphi_i$}.
As $\calB^0$ is a hierarchical generator (with an empty lineage) repeated
application of Lemma \ref{l:single_ref} implies that $\calH_{N}$ is a
hierarchical generator with lineage  $\calL = \{\vphi_0,\dots,\vphi_{N-1}\}$.
What is more interesting is that any hierarchical generator can be obtained in
this way, thus justifying the more ``abstract'' definition of lineages
given in Definition~\ref{d:lineage}.
%%
%%%%%%%%%%%%%%%%%%%%%%%%%%%%%%%%%%%%%%%%%%%%%%%%%%%%
\begin{lemma}[Lineages and refinements]
 A sequence of single refinements starting from $\calB^0$ yields
 a hierarchical generator 
 and reciprocally any hierarchical generator can be obtained by a 
 sequence of single refinements.
\end{lemma}
%%%%%%%%%%%%%%%%%%%%%%%%
\begin{proof}
The first statement of the lemma was shown in the previous paragraph above.
For the second statement let $\bar\calH$ be a hierarchical generator 
and $\bar\calL$ its lineage.
Define $\calM_i = \bar\calL^i$ for $i\in(0:\depth(\bar\calL)-1)$.
Let $\calH_0=\calB_0$ and $\calH_{i+1}$ the output of
\Call{Refine}{$\calH_i$,$\calM_i$}.
Clearly $\calH_0$ is a hierarchical generator, 
$\calL_0=\emptyset$, $\calM_0\subset\calH_0$.
%, and it turns out that $\calH_1 = \calM_0$.
%%
Assume $\calH_n$ is a hierarchical generator, 
$\calL_n=\dcup_{i\in[0:n-1]}\calM_i$ 
and $\calM_n\subset\calH_n$.
Using Lemma \ref{l:refine0} it follows that
$\calH_{n+1}$ is a hierarchical generator, 
$\calL_{n+1}=\dcup_{i\in[0:n]}\calM_i$.
Now $\calM_{n+1} = \bar\calL^{n+1} \subset \ch{\bar\calL^{n}} = \ch{\calM_n}$,
then $\calM_{n+1}\subset\calC_{n+1}$ as $\calL_{n+1}^{n+1}=\emptyset$ it
follows that $\calM_{n+1}\subset\calH_{n+1}$.
Thus we have shown that $\calH$ is a hierarchical generator
that was obtained by a sequence of single refinements
with lineage $\calL=\dcup_{i\in(0:\depth(\bar\calL)-1)}\calM_i =
\bar\calL$, so using Lemma \ref{l:bijection} $\calH=\bar\calH$.
\end{proof}
%%%%%%%%%%%%%%%%%%%%%%%%%%%%%%%%%%%%%%%%%%%%%%%%%%%%

\subsection{Origin of new functions}
\pedro{Given a refinement $\calH_*$ of $\calH$, one intuitively expects that any function in the refiner set $\calR = \calL_* \setminus \calL$ has been originated by a refined function in $\calH$, i.e., from $\calR \cap \calH$. Similarly, if a function is in the refiner set $\calR$ it must have generated a new function in $\calH_*$.}
%\pedro{After a call to $\Call{Refine}{\calH,\calM}$ with $\calM \subset \calH$ a new hierarchical generator $\calH_*$ is obtained. The refinement set $\calR = \calL_* \setminus \calL$ thus coincides with $\calM$. 
%Therefore, any function in the refiner set $\calR$ is a refined function from $\calH$, and one intuitively expects that if a function is in the refiner set, it must have generated a new function in $\calH_*$.}
%One intuitively expects that any function in the refiner set $\calR$ is somehow
%originated by a refined function in \calH, and similarly that if a 
%function is in the refiner set, it will somehow generate a new function in 
%$\calH_*$.
\pedro{This relation is important to obtain complexity results relating the number of marked functions and the dimension of the hierarchical spaces in the context of an adaptive loop, and}
%We study these expectations below.
we elaborate on this below.

%%%%%%%%%%%%%%%%%%%%%%%%%%%%%%%%%
\begin{lemma}[New function cause]\label{l:ref_cause}
Let \calH be a hierarchical generator, let
$\calH_*\refi\calH$ with refiner set \calR,
and $\calM:=\calR\cap\calH$ then
\begin{enumerate}[(i)]
  \item\label{l:ref_cause:1} if $\psi_0\in\calR$ there exists $k\geq 0$
and a sequence $(\psi_0,\dots,\psi_k)$ with 
$\psi_k\in\calM$, 
$\psi_j\in\calR\setminus\calM$ % for $j\in[0:k-1]$
and
$\psi_j\in\ch{\psi_{j+1}}$, for $j\in[0:k-1]$
\item\label{l:ref_cause:2}
 if $\vphi_*\in \calH_*\setminus\calH$ then there is
$\psi\in\calR$ such that $\vphi_*\in \ch{\psi}$
 \item\label{l:ref_cause:3}
 if $\vphi_*\in \calH_*\setminus\calH$ then there is $k\geq 1$
and $\vphi\in \calM$ such that $ \vphi_*\in\ch[k]{\vphi}$.
 \end{enumerate}
\end{lemma}
\begin{proof}
To show result~\ref{l:ref_cause:1} we proceed by induction on the level of 
$\psi_0$.
If $\l_{\psi_0}=0$ Lemma \ref{l:ref_identities}\ref{l:ref_identities:R} implies 
that
$\psi_0\in\calM$ so the result follows with $k=0$.
Now assume that statement \ref{l:ref_cause:1} holds for
any function of level $n$, and let $\l_{\psi_0}=n+1$.
Again using Lemma \ref{l:ref_identities}\ref{l:ref_identities:R} there are two 
possibilities.
Either $\psi_0\in\calM$ and the result follows with $k=0$ or
$\psi_0\in\ch{\calR}$.
In the latter case, there is $\psi_1\in\calR$ with
$\psi_0\in\ch{\psi_1}$, whence $\l_{\psi_1}=n$ and 
the inductive assumption yields the desired assertion.

Assertion~\ref{l:ref_cause:2} is a direct consequence of 
Lemma~\ref{l:ref_identities}\ref{l:ref_identities:HrminusH}, 
and~\ref{l:ref_cause:3} follows from~\ref{l:ref_cause:1} 
and~\ref{l:ref_cause:2}.
\end{proof}

%%%%%%%%%%%%%%%%%%%%%%%%%%%%%%%%%%%%%%%%%%%%%%%%%%%%%%%%%%%%%%%
\section{Linear Independence}\label{s:linear_independence}
The hierarchical generators together with the \Call{Refine}{} procedure
of Algorithm \ref{a:refine0} give a remarkably 
simple mechanism to obtain spaces with the required local resolution 
[cf.~Requirement in 
Properties~\ref{p:refine}:\ref{p:refine:res},\ref{p:refine:simple}].
But, as we can see in the next example
it may not give automatically the linear independence stated in
Property~\ref{p:generators}\ref{p:generators:li}.
\begin{example}[Generator not linearly independent]\label{ex:generator_not_li}
Consider $m=3$ with $d=1$, let 
$\calL=\{\vphi^{0}_{-2}, \vphi^{0}_{0}\}$ 
So $\calH=\{\vphi^{0}_{-1}, \vphi^{1}_{-2} , \vphi^{1}_{-1}, \vphi^{1}_{0}, 
\vphi^{1}_{1}\}$.
Clearly, $\vphi^{0}_{-1}$ can be spanned as linear combination of
$\{\vphi^{1}_{-2} , \vphi^{1}_{-1}, \vphi^{1}_{0}, 
\vphi^{1}_{1}\}$, thus $\calH$ is not linearly independent.
\end{example}

In this section we deal with transformations that can be
applied to a generator to ensure it is a basis.

\begin{definition}[Hierarchical basis]\label{d:hb} 
We say that $\calH$ is a \emph{hierarchical basis} if it is a linearly 
independent hierarchical generator.
\end{definition}
%%

% \seba{%
% In light of Lemma~\ref{l:ml_span}\ref{l:ml_span:equiv} we have that
% \begin{lemma}[Basis equivalence]\label{l:basis_equiv}
%  A generator \calH is a hierarchical basis if and only if
%  for each $\vphi\in\calH$ it follows that
%  $\ch{\vphi}\not\subset\spn{(\calH\setminus\{\vphi\})}$.
% \end{lemma}
% }

From Lemma~\ref{l:pu} %with Definition~\ref{d:hb} 
we immediately obtain the following.
\begin{lemma}[Unique positive partition of unity] 
Let $\calH$ be a hierarchical basis, then for
each $\vphi \in \calH$, there exists a unique
constant $c_\vphi > 0$ such that
$\sum_{\vphi \in \calH} c_\vphi \vphi = 1$.
\end{lemma}

%\subsection{New descendants}
%\paragraph{New descendants.}
One interesting property of a hierarchical basis which does not hold for 
arbitrary hierarchical generators is that every function
in \calH that is refined has a descendant in $\calH_*\setminus\calH$.

%%%%
\begin{lemma}[Refined function effect]\label{l:ref_effect}
Let $\calH$ be a hierarchical basis, let $\calH_*\refi\calH$ with refiner set $\calR = \calL_* \setminus \calL$. 
\pedro{If $\vphi\in\calM = \calR \cap \calH$, then $\descs\vphi \cap (\calH_*\setminus\calH)\neq \emptyset$, i.e.,} 
there is $k>0$ such that there exists $\vphi_*\in (\calH_*\setminus\calH)\cap\ch[k](\vphi)$.
\end{lemma}
\begin{proof}
We prove the result by contradiction.
Let $\vphi\in\calM$ and suppose that 
\pedro{$\descs\vphi \cap (\calH_*\setminus\calH)= \emptyset$, then $\descs\vphi 
\cap \calH_* \subset \descs\vphi \cap \calH$. Besides, from 
Lemma~\ref{l:H_L_props}\ref{l:H_L_props:L_from_H}, 
	as $\vphi\in\calL_*$,
	we have that
	$\vphi\in\spn(\descs\vphi \cap \calH_*)
	\subset \spn(\descs\vphi \cap \calH) \subset \spn(\calH \setminus \{\vphi\})$.
}
This implies that \calH is linearly dependent which contradicts the assumption.
\end{proof}

\subsection{Linearly independent refinement}\label{sec:liref}
\pedro{We want to work with hierarchical spaces, in particular with those appearing in an adaptive process where some functions are selected and refined to add local resolution.
It turns out that the refinement procedures defined thus far produce hierarchical generators that may not be linearly independent.
Linear independence is desirable in order to fulfill Property~\ref{p:generators}\ref{p:generators:li}, to avoid redundancy and ill-posedness of the resulting (non-)linear systems.}
Removing redundant functions may be demanding task and may lead to generators that are not hierarchical.
One interesting approach is to consider linearly independent
refinements of hierarchical generators while investigating the following 
questions.
\begin{enumerate}
 \item Given a hierarchical generator, which is the smallest
 linearly independent refinement? Does it exist?
 \item If it exists, can we characterize it in terms of
 a property of the lineage?
 \item \pedro{Does it span the same space or a larger one? How much larger?}
 \item Can we provide a simple constructive procedure to find it?
\end{enumerate}

The first question can be mathematically written as follows:
Given a generator \calH, find the smallest element of the family
\begin{equation}\label{e:lin_ind_H}
\textgoth{B}(\calH) = \{\calH_* : \calH_*\refi\calH \text{ and } \calH_* 
\text{ is linearly independent}\}.
\end{equation}
It is still an open question whether in general the 
minimal  set of $\textgoth{B}(\calH)$ 
is a singleton, empty or larger. 
This matter, which is intimately related with the characterization 
of linear independence in terms of the lineage properties, is part of an 
ongoing 
work. %[see Section~\ref{s:future_work}].

An important point is that the condition imposed in \eqref{e:lin_ind_H}
can be replaced by one stronger than just linear independence.
The mathematical framework in which to develop 
a successful theory in the light of Properties~\ref{p:refine} and 
\ref{p:generators} can be summarized as follows.
\begin{enumerate}[(i)]
\item\label{en:1} State a condition $\textgoth{A}$ for the hierarchical 
generators
that implies linear independence
\item Consider the family
$\textgoth{A}(\calH) = \{\calH_* : \calH_*\refi\calH \text{ and }  
\calH_* 
\text{ satisfies condition $\textgoth{A}$}\}$.
\item \label{en:3} Show that the smallest element of $\textgoth{A}(\calH)$ 
exists.
\item \label{en:4}Show that the cardinality of this element is not much larger 
than 
$\dim1(\spn\calH)$.
\item Provide a simple method to construct this element.
\end{enumerate}
We remark that \ref{en:1} only asks for a sufficient condition for linear 
independence, thus the smallest refinement of \ref{en:3} may yield a basis 
bigger than the dimension of \calH.
Thus \ref{en:4} is an important restriction on the condition $\textgoth{A}$.

\subsection{A sufficient condition. Absorbing Generator}
A sufficient condition for  linear independence of a generator can be
obtained following the intuition that 
if a function in \calH is totally overlapped by finer
functions in \calH, then that function is very likely  
redundant. This idea with a different language can be ascribed to the
work of \cite{VGJS2011}. 
We now explore this concept in the framework described in 
Section~\ref{sec:liref}, 
presenting a sufficient condition for a hierarchical generator to be linearly 
independent.

\begin{definition}[Absorbing Generator]\label{d:lineage2}
	A hierarchical generator \calH is called \emph{absorbing} if
%	any $\vphi \in \calC$ with $\bbI(\vphi)\subset\bbI(\calL^{\l_\vphi})$ belongs 
%to $\calL$.
	for any $\vphi\in\calC$ such that $\bbI(\vphi)\subset\bbI(\calL)$ 
	it holds that $\vphi\in\calL$. 
	\pedro{In other words, $\calH$ is absorbing if there is no $\vphi \in \calH$ such that $\bbI(\vphi)\subset\bbI(\calL)$.}
\end{definition}

%%
%\subsection{The sufficient condition}
%%
For an absorbing generator we have that each B--spline in $\calH$ 
overlaps an active cell of its own level. We state this more precisely as follows.
\begin{lemma}[Overlap of active cells]\label{l:overlapactive}
If $\calH$ is absorbing then $\calH \subset \bbB(\calA)$, where $\calA$ denotes the set of active cells corresponding to $\calH$, according to Definition~\ref{d:activecells}.
\end{lemma}
\begin{proof} 
Let $\calH$ be an absorbing generator, and let $\vphi \in \calH$. We 
want to prove that $\vphi \in \bbB(\calA)$, which is equivalent to 
$\bbI(\vphi)\cap \calA \neq \emptyset$. Assume, on the contrary, that 
$\bbI(\vphi)\cap \calA = \emptyset$.
	Since by definition $\calA = \bbI(\calC)\setminus \bbI(\calL)$, this implies 
that $\bbI(\vphi) \subset \bbI(\calL)$, which due to the fact that $\calH$ is 
absorbing, implies that $\vphi \in \calL$, which contradicts the assumption 
that 
$\vphi \in \calH$. The assertion thus follows.
\end{proof}

\begin{lemma}[Linear independence]\label{l:li}
Every absorbing hierarchical generator is linearly independent, 
and thus a hierarchical basis.
\end{lemma}
%%%%%
\begin{proof}
	Let $\calH$ be a hierarchical generator and assume that
	$\sum_{\vphi\in\calH} \alpha_\vphi \vphi = 0$. Then this function vanishes in 
every active cell, i.e., 
	\[
	\sum_{\vphi\in\calH} \alpha_\vphi \vphi = 0 \quad\text{in $I$},
	\qquad \text{for every } I \in \calA = \bigcup_{\ell=0}^{\depth(\calH)} 
\calA^\l.
	\]
Then, we have
$0=\sum_{\vphi\in\calH^0} \alpha_\vphi \vphi$ in each $I\in \calA^0$. 
Due to Lemma~\ref{l:H_L_props:H_overlap_by_L1} all functions in 
$\cup_{\l=1}^{\depth(\calH)} \calH^\l$ vanish in all $I \in \calA^0$, and since 
$\calB^0$ are locally linear independent, it follows that $\alpha_\vphi=0$ for 
all $\vphi\in\calH^0\cap\bbB(\calA^0)$. From Lemma~\ref{l:overlapactive}, 
$\calH^0 \cap \bbB(\calA^0) = \calH^0$ and 
$\alpha_\vphi=0$ for each $\vphi\in\calH^0$.
Arguing by induction, we conclude that $\alpha_\vphi=0$ for all 
$\vphi\in\calH$, and the assertion follows.
\end{proof}

\begin{definition}[Absorbing basis]\label{d:hbasis}
An absorbing hierarchical generator is called an 
absorbing hierarchical basis, or merely an \emph{absorbing basis}.
\end{definition}
\begin{remark}[Absorbing is not necessary for linear independence]\label{R:absorbing-not-necessary}
It is worth noticing that the absorbing condition is a sufficient condition for linear independence, but not necessary. In fact
%\begin{example}[Not absorbing hierarchical basis]\label{ex:not_abs_basis}
Consider $m=2$ with $d=1$, let 
$\calL=\{\vphi^{0}_{-1}, \vphi^{0}_{0}, \vphi^{1}_{-1}, \vphi^{1}_{1},\}$ 
So $\calH=\{\vphi^{1}_{0}, 
            \vphi^{2}_{-1} , \vphi^{2}_{0}, 
            \vphi^{2}_{2} , \vphi^{2}_{3}\}$.
This $\calH$ is linearly independent, but not absorbing.
%\end{example}
\end{remark}

\subsection{The absorbing refinement}

Given a hierarchical generator $\calH$, let us consider the family
\begin{equation}\label{e:abs_ref}
\textgoth{A}(\calH) = \{\calH_* : \calH_*\refi\calH, \calH_* \text{ is 
absorbing}\}.
\end{equation}
Algorithm \ref{a:make_absorbing} constructs the least element of this
family.
%%%%%%%%%%
\begin{algorithm}[h!]
\caption{Absorbing Refinement Algorithm}\label{a:make_absorbing}
\begin{algorithmic}[1]
\Function{AbsRefine}{$\calH$}
\State $\tilde{\calC}^{0} = \calB^0$; $\tilde{\calL}=\emptyset$
\Comment{At this point $\tilde{\calH} = \calB^0$}
\For{$\l=0$ to $\depth(\calH)-1$}%\Comment{$N$ is the depth of \calH}
\State $\calM = \{\vphi\in \tilde{\calC}^{\l}\setminus \calL^{\l} : 
\bbI(\vphi)\subset \bbI(\calL^{\l})
%\vphi\ab\calL^{\l}
\}$
\State \Call{Refine}{$\tilde{\calH},\calM\cup\calL^{\l}$}
\Comment{Now $\tilde{\calL}^{\l} ={\calL}^{\l}\cup\calM$  and 
$\tilde{\calC}^{\l+1} = 
\ch{\tilde{\calL}^{\l}}$}
\EndFor
\State\Return{$\tilde{\calH}$}
\EndFunction
\end{algorithmic}
\end{algorithm}
%%%%%%%%%%%%

The properties of this algorithm are summarized in the following Lemma.

\begin{lemma}[Properties of \Call{AbsRefine}{}]\label{l:abs_ref}
Let $\calH$ be a hierarchical generator, a call to  
$\Call{AbsRefine}{\calH}$ returns in $\tilde\calH$
the least element of the family $\textgoth{A}(\calH)$ 
from~\eqref{e:abs_ref}, 
i.e., the smallest absorbing basis which is a refinement of $\calH$.
Furthermore, we have that 
\begin{enumerate}[(i)]
\item \label{l:abs_ref:absd} 
If $\psi\in\calR:=\tilde\calL \setminus \calL$, then 
$\bbI(\psi)\subset\bbI(\calL^{\l_{\psi}})$.
\item \label{l:abs_ref:depth}
 $\depth{\calH} = \depth{\tilde\calH}$.
\item\label{e:gap1} If $\psi\in\calR$ there is $k>0$ 
such that 
$\ovch[]{\psi}{k}{\calH}\not=\emptyset$.
\end{enumerate}
\end{lemma}
%%%%
\begin{proof}
At the start of the loop $\tilde\calH=\calB^0$ is a hierarchical generator,
inside the loop $\depth(\calH)$ valid iterated calls to \Call{Refine}{} are 
made.
Thus Lemma~\ref{l:refine0} implies that $\tilde\calH$ is a
refinement of \calH and
\begin{equation}\label{e:defi}
\tilde\calL^{\l} = \{\vphi\in \tilde{\calC}^{\l} : 
\bbI(\vphi)\subset \bbI(\calL^{\l})
%\vphi\ab\calL^{\l}
\},
\end{equation}
so that $\bbI(\tilde\calL^\l) = \bbI(\calL^\l)$. 
% $\bbI(\vphi) \in \bbI(\tilde\calL^{\l})$ implies $\bbI(\vphi) \in 
% \bbI(\calL^{\l})$.
Therefore, if $\vphi \in \tilde\calL^\l$ and $\bbI(\vphi) \subset 
\bbI(\tilde\calL^\l)$ then $\vphi \in \tilde\calL^\ell$ and $\tilde{\calH}$ is 
thus absorbing.
%where we get that if $\vphi\ab\tilde\calL^{\l}$ then 
%$\vphi\ab\calL^{\l}$, thus if
%$\vphi\in \tilde{\calC}^{\l}$ and $\vphi\ab\tilde\calL^{\l}$
%it follows that $\vphi\in\tilde\calL^{\l}$ so that
%$\tilde\calL$ is absorbing. 
Hence $\tilde\calH$ is an absorbing refinement of \calH.

Now we show that it is in fact the smallest of such refinements.
To see this take $\calH_*$ another absorbing refinement of \calH.
If $\vphi\in\tilde{\calL}^0$ then 
$\vphi\in\tilde{\calC}^0={\calB}^0=\calC_*^0$
and from \eqref{e:defi} $\vphi\ab\calL^0$.
Then as $\calH_*$ is an absorbing refinement it follows that
$\vphi\in\calL_*^0$, thus $\tilde{\calL}^0\subset\calL_*^0$.
We now proceed by induction.
Suppose we have shown that $\tilde{\calL}^n\subset\calL_*^n$, so
$\tilde{\calC}^{n+1}\subset\calC_*^{n+1}$ and take
$\vphi\in\tilde{\calL}^{n+1}$ then 
$\vphi\in\calC_*^{n+1}$ but also $\bbI(\vphi) \subset \bbI(\calL^{n+1})\subset 
\bbI(\calL_*^{n+1})$.
Then, as $\calH_*$ is absorbing it follows that
$\vphi\in\calL_*^{n+1}$, whence $\tilde{\calL}^{n+1}\subset\calL_*^{n+1}$.
Summarizing, we have shown that $\calH_*\refi\tilde\calH$, so $\tilde\calH$ is 
in fact 
the smallest set in $\textgoth{A}(\calH)$ from~\eqref{e:abs_ref}.
Assertions~\ref{l:abs_ref:absd} and~\ref{l:abs_ref:depth} immediately follow.
\seba{%
From part \ref{l:abs_ref:absd}, $\bbI(\psi)\subset\bbI(\calL^{\l_{\psi}})$,
thus for any $I\in\bbI(\psi)$ there is $\eta\in\calL$ such that
$I\in\bbI(\eta)$.
Using Lemma~\ref{c:H_L_props}\ref{c:H_L_props:L_overlap_by_dsc},
$I\in\cup_{k>0}\bbI^{-k}(\ch[k]{\eta}\cap\calH)$,
thus there must exist $k>0$ and $\phi\in\ch[k]{\eta}\cap\calH$
such that $I\in\bbI^{-k}(\phi)$, thus 
$\vphi\in\ovch[]{\psi}{k}{\calH}$ and part \ref{e:gap1} follows.}
\end{proof}

\pedro{
\begin{remark}[Comparison with other hierarchical basis]\label{r:coincidencia}
	Our concept of absorbing hierarchical basis coincides with the concept of hierarchical basis from~\cite{BG2016}. In fact, given an absorbing hierarchical basis $\calH$ with corresponding lineage $\calL$, after defining $\omega_\ell = \cup_{\vphi \in \calL^\ell} \supp \vphi$, it is straightforward to check that the definition from~\cite{BG2016} leads to the same space.
\end{remark}
}

\pedro{The next result will be important when studying the \emph{gap} of a hierarchical generator which is the subject of the following section.}

\begin{lemma}[New function]\label{l:new_max}
Let $\calH$ be a hierarchical generator, and $\tilde\calH$ the result of a call to  
$\Call{AbsRefine}{\calH}$.
If $\vphi\in\tilde\calH\setminus\calH$ and $k>0$,
then for each $\eta\in\ovch{\vphi}{-k}{\calB}$  
there exists $\zeta\in\calH$ and $k'\geq k$ such that
$\eta\in\ovch{\zeta}{-k'}{\calB}$.
%% then $\vphi\ab\cup_{k\geq 0} \calH^{\l_\vphi+k}$.
\end{lemma}

\begin{proof}
	From Lemma~\ref{l:ref_identities}\ref{l:ref_identities:HrminusH}, if $\vphi 
\in \tilde{\calH}\setminus \calH$, there exists $\psi \in \calR$ such that 
$\vphi \in \ch[]{\psi}$, so that Lemma~\ref{l:over_ch} yields 
$\ovch[]{\vphi}{-k}{\calB} \subset \ovch[]{\psi}{-k+1}{\calB}$. 
	Therefore, if $\eta \in \ovch[]{\vphi}{-k}{\calB}$, we have $\eta \in 
\ovch[]{\psi}{-k+1}{\calB}$ and~\eqref{ovch} implies the existence of $I \in 
\bbI(\psi)\cap \bbI^{k-1}(\eta)$. Since $\psi \in \calR$, $\bbI(\psi) \subset 
\bbI(\calL^{\l_\psi})$ so that $I \in \bbI(\calL^{\l_\psi}) \cap 
\bbI^{k-1}(\eta)$.
	Consequently, there exists $\hat \psi \in \calL^{\l_\psi}$ with $I \in 
\bbI(\hat\psi)\cap \bbI^{k-1}(\eta)$.
	By Lemma~\ref{l:H_L_props}\ref{l:H_L_props:L_from_H}, $\hat{\psi} \in 
\spn(\descs(\hat{\psi})\cap \calH)$, whence there exists $\zeta \in 
\descs(\hat{\psi})\cap \calH$ and $j \ge 1$ such that $\bbI(\zeta) \cap 
\ch[j]{I} \neq\emptyset$, hence, $\bbI(\zeta) \cap 
\bbI^{k-1+j}(\eta)\neq\emptyset$, which implies that $\eta \in 
\ovch[]{\eta}{-k'}{\calB}$ with $k' = k-1+j \ge k$, due to~\eqref{ovch}.
\end{proof}

%%%%%%%%%%%%%%%%%%%%%%%%%%%%%%%%%%%%%%%%%%%%%%%%%%%%%
\section{Overlapping and gap constraint}\label{s:gap}
%%%%%%%%%%%%%%%%%%%%%%%%%%%%%%%%%%%%%%%%%%%%%%%%%%%%%

We now %In this section we 
deal with Property~\ref{p:generators}\ref{p:generators:overlap}, i.e., we address the issue of controlling the level difference
of overlapping functions in a given generator.

%\subsection{Gap of a generator}
To measure the function overlapping in a generator we assign
a number, called the \emph{gap}, to each function in the generator or basis.
This number, associated to each function in the generator, 
measures the level difference with the coarsest overlapping function.

\begin{definition}[Gap of a function]
Let \calH be a hierarchical generator and $\vphi\in\calH$,
then we define the \emph{gap} of $\vphi$ in \calH as 
$\fgap{\vphi} := \sup\{g\in\Z: \ovch{\vphi}{-g}{\calH} 
\not= \emptyset\}$.
\end{definition}

\pedro{Computing the gap of a function can be rather expensive, but as we will see below, it will never be necessary to perform such a computation; see Remark~\ref{R:never compute gap}.}

%%%%%%%%%%%%%%%%
\begin{lemma}[Properties of the gap]\label{l:gap}
Let \calH be a hierarchical generator and $\vphi\in\calH$ then
the following properties are satisfied
\begin{enumerate}[(i)]
 \item $\fgap{\vphi}\in\Z^+_0$.
 \item \label{l:gap:s1}$\ovch{\vphi}{-\fgap{\vphi}}{\calH} \not=\emptyset$.
 \item \label{l:gap:s2}$\ovch{\vphi}{-g}{\calH} =\emptyset$ for 
$g>\fgap{\vphi}$.
 \item \label{l:gap:s3} If $\ovch{\vphi}{-g}{\calH} \not=\emptyset$ then 
$g\leq\fgap{\vphi}$.
\end{enumerate}
\end{lemma}
%%%%%%%%%%%%%%%%
%%
\begin{proof}
As $\vphi\in\ovch{\vphi}{0}{\calH}$ the set on which the supremum is taken 
is not empty and $\ovch{\vphi}{-g}{\calH}$ is empty for $g>\l_{\vphi}$ so
the set is bounded above by $\l_{\vphi}$.
Then every $\vphi$ has a non negative gap assigned.
%%%
The last two items follow directly from the definition of supremum.
\end{proof}
%%%%%%%%%%%%%%%%

The next result states that the gap of a $k$-th descendant of a function $\vphi$ is bounded by the gap of $\vphi$ plus $k$.
 %%%%%%%%%%%%%%%%
\begin{lemma}[Gap of a descendant]\label{p:gap_ch}
Let \calH be a hierarchical generator,
$\vphi\in\calH$ and $k\geq0$,
then for $\psi\in\ch[k]{\vphi}\cap\calH$ we have that $k\leq\fgap{\psi}\leq 
\fgap{\vphi}+k$.
\end{lemma}
%%%%%%%%%%%%%%%%
%%%
\begin{proof}
From Lemma~\ref{l:over_ch}, for $\psi\in\ch[k]{\vphi}\cap\calH$ we get that
$\{j : \ovch{\psi}{-j}{\calH} \not=\emptyset \} 
\subset \{j : \ovch{\vphi}{-j+k}{\calH} \not=\emptyset \} = 
\{j+k : \ovch{\vphi}{-j}{\calH} \not=\emptyset \}$,
and taking supremum we get that $\fgap{\psi}\leq \fgap{\vphi}+1$.
\end{proof}

Observe that the gap of a $k$-th descendant of $\vphi$ can actually take
any value between $k$ and $\fgap{\vphi}+k$.

%%
%\subsection
\paragraph{Refinement and gap}
A refinement process may change the gap of the functions in a generator.
The following result states
that if a function stays in the generator after refinement, 
its gap does not increase, and in fact it can actually 
decrease.

%%%%%%%%%%%%%%%%%%%%%%
\begin{lemma}[Refinement and gap]\label{l:ref_and_gap}
 Let $\calH$ be a hierarchical generator and $\calH_* \refi \calH$, then:
 \begin{enumerate}[(i)]
  \item \label{l:ref_and_gap:new} 
  If $\vphi\in\calH_*$ there exists $k\geq\fgap[\calH_*]{\vphi}$,
  such that $\ovch{\vphi}{-k}{\calH}\not=\emptyset$.
\item \label{l:ref_and_gap:stay}
 If $\vphi\in\calH_*\cap\calH$,
 then $\fgap[\calH_*]{\vphi} \le \fgap[\calH]{\vphi}$.
 \end{enumerate}
\end{lemma}
%%%%%%%%%%%%%%%%%%%%%%
\begin{proof}
Let $\vphi\in\calH_*$, due to Lemma \ref{l:gap}\ref{l:gap:s1} there exists
$\eta\in\ovch[]{\vphi}{-\fgap[\calH_*]{\vphi}}{\calH_*}$, 
i.e., $\eta\in\calH_*$ % , $\l_{\vphi}-\l_{\eta}=\fgap[\calH_*]{\vphi}$ 
and $\bbI^{\fgap[\calH_*]{\vphi}}(\eta) \cap \bbI(\vphi) \neq \emptyset$.
If $\eta\in\calH$ then $\eta \in \ovch[]{\vphi}{-\fgap[\calH_*]{\vphi}}{\calH}$ 
and $\ovch[]{\vphi}{-k}{\calH} \neq\emptyset$ for $k=\fgap[\calH_*]{\vphi}$.
If $\eta\in\calH_*\setminus\calH$ then from 
Lemma~\ref{l:ref_cause}\ref{l:ref_cause:3}, 
there exists an ancestor $\xi\in\calM\subset\calH$ and $j\ge 1$, 
such that $\eta \in \ch[j]{\xi}$ so that $\bbI(\eta) \subset \bbI^j(\xi)$. 
Therefore, $\bbI(\vphi) \cap \bbI^{\fgap[\calH_*](\vphi)+j}(\xi) \neq\emptyset$ 
and thus $\xi \in \ovch[]{\vphi}{\fgap[\calH_*](\vphi)+j}{\calH}$ and the 
assertions follow.
\end{proof}
%%%%%%%%%%%%%%%%%%%%%%

We now define the \emph{gap of a generator}.

\begin{definition}[Gap of a hierarchical generator]\label{d:gapH}
Given a hierarchical generator $\calH$ we define its gap as
$\gap{\calH} = \max\{\fgap{\vphi} : \vphi \in \calH \}.$
\end{definition}

%\subsection{Gap of an absorbing refinement}
%%
In the next result we prove that
the process of making a hierarchical generator absorbing, through
taking its smallest absorbing refinement (with Algorithm \Call{AbsRefine}{}) 
does not increase its gap. 
\begin{proposition}[\Call{AbsRefine}{} does not increase the gap]\label{p:max_gap}
Let $\calH$ be a hierarchical generator and let $\tilde{\calH}$ be the result of a call to $\Call{AbsRefine}{\calH}$ from 
Algorithm~\ref{a:make_absorbing}.
Then $\gap{\tilde{\calH}}\leq\gap{\calH}$.
\end{proposition}
\begin{proof}
From Lemma \ref{l:ref_and_gap}\ref{l:ref_and_gap:stay} 
if $\vphi\in\tilde{\calH}\cap\calH$ then 
$\fgap[\tilde\calH]{\vphi}\leq\fgap{\vphi}\leq\gap{\calH}$.
If $\vphi\in\tilde{\calH}\setminus\calH$ 
using Lemma \ref{l:ref_and_gap}\ref{l:ref_and_gap:new}
there is $k\geq\fgap[\tilde\calH]{\vphi}$ and 
$\eta\in\ovch{\vphi}{-k}{\calH}$.
Now using Lemma~\ref{l:new_max}
there is $\zeta\in\calH$, 
and $k'\geq k$ such that
$\eta\in\ovch{\zeta}{-k'}{\calB}$, whence $\eta\in\ovch{\zeta}{-k'}{\calH}$.
From Definition~\ref{d:gapH} and Lemma~\ref{l:gap}\ref{l:gap:s3} we get
$\gap{\calH}\geq\fgap{\eta}\geq k'\geq k\geq\fgap[\tilde\calH]{\vphi}$.
We have shown that for any $\vphi\in\tilde{\calH}$, 
$\fgap[\tilde\calH]{\vphi}\leq \gap{\calH}$ and the result follows.
\end{proof}

\begin{remark}[\Call{AbsRefine}{} could decrease the gap]
The $\gap{\tilde\calH}$ is not necessarily equal to the
gap of \calH, the gap can actually decrease after 
calling \Call{AbsRefine}{}. 
The hierarchical generator from Remark~\ref{R:absorbing-not-necessary}, has gap equal to 1, and after $\Call{AbsRefine}{}$ we obtain the generator $\calH = \{ \vphi_{-1}^2, \vphi_{0}^2, \vphi_{1}^2, \vphi_{2}^2\}$ which has gap 0.
\end{remark}
%%

%%%%%%%%%%%%%%%%%%%%%%%%%%%%%%%%%%%%%%%%%%%%%%%%%%%%%%%%%%%%%%%%%%%%%%

\section{Function refinement with gap constraint}\label{S:refinement with gap constraint}
%%%
Following the requirement in 
Property~\ref{p:generators}\ref{p:generators:overlap},
the refinement of a hierarchical generator $\calH$ should maintain
the gap bounded above by a fixed given positive integer $g$.
We explore on the effect of refinement on the gap of a generator in the following Lemma.
%%
%%%%%%%%%%%%%%%%%%%%%%%%%%%%%%%%%%%%%%%%%%%%%%%%%%%%%%%%%55
\begin{lemma}[Single refinement and gap]\label{l:single_ref_gap}
Let $\calH$ be a hierarchical generator, $\vphi\in\calH$ and $\calH_*$ its 
refinement 
after a call to $\Call{SingleRefine}{\calH,\vphi}$ of Algorithm 
\ref{a:single_ref}.
Then, if $\vphi_* \in \calH_*$,
%Let $\vphi_*\in\calH_*$, then 
\[
\fgap[\calH_*]{\vphi_*}\leq
\begin{cases} \fgap{\vphi} + 1, \text{ if } \vphi_*\in\calH_*\setminus\calH,\\ 
 \fgap{\vphi_*} \text{ if } \vphi_*\in\calH_*\cap\calH.
\end{cases}
\]

\end{lemma}
%%%%%%%%%%%%%%%%%%%
\begin{proof}
	Recall from Lemma~\ref{l:good_def_sr} that
$\calH_* = (\calH \setminus \{\vphi\}) \dcup (\ch{\vphi}\setminus{\calC})$,
from where
$\calH_*\setminus\calH =\ch{\vphi}\setminus{\calC}$ and 
$\calH_*\cap\calH = (\calH \setminus \{\vphi\})$.

	If $\vphi_*\in\calH_*\setminus\calH = \ch{\vphi}\setminus{\calC}$,
then
$\ovch[]{\vphi_*}{-k}{\calH_*}\subset 
\ovch{\vphi_*}{-k}{\calH}\setminus\{\vphi\}$
for all $k>0$.
Also, by Lemma~\ref{l:over_ch}
$\ovch[]{\vphi_*}{-k}{\calH}\subset\ovch{\vphi}{-k+1}{\calH}$, then 
Lemma~\ref{l:gap}\ref{l:gap:s2} implies that 
$\ovch[]{\vphi_*}{-k}{\calH_*}=\emptyset$ for any $k>\fgap{\vphi}+1$, so that
$\fgap[\calH_*]{\vphi_*}\leq\fgap{\vphi} + 1$.
The case of $\vphi_* \in \calH_* \cap \calH$ follows from 
Lemma~\ref{l:ref_and_gap}\ref{l:ref_and_gap:stay}.
\end{proof}

\seba{
It is simple to construct an example of a generator with gap $g$ and a 
refinement that increases its gap.
Indeed, this can happen by refining a single function.
The possibility of \Call{SingleRefine}{} to increase the gap of a generator
makes it \emph{necessary} to find a new mechanism to obtain the smallest 
refinement that ensures the bound on the gap.  Thus, given a hierarchical 
generator \calH and $\vphi\in\calH$, consider the set
\begin{equation}\label{e:refi_gap_set}
\textgoth{R}_g(\calH, \vphi) =\{\calH_ : \calH_*\refi\calH, \vphi\in\calL_* 
\text{ and } 
\gap{\calH_*}\leq g\},
\end{equation}
where $g$ is a positive integer that we consider fixed from now on.
We would like to find the smallest element of  \eqref{e:refi_gap_set}.}
\pedro{In order to do it, we first observe the following.}

\pedro{
	\begin{remark}\label{R:necessity-of-refining-ancestors}
		Let $\calH$ be a hierarchical generator with $\gap(\calH)\le g$, let $\vphi \in \calH$ and let $\calH_*$ be the refinement obtained after a call to $\Call{SingleRefine}{\calH,\vphi}$. Then, as a consequence of Lemmas~\ref{l:single_ref_gap} and~\ref{l:over_ch}, we have:
		\begin{enumerate}[(i)]
			\item If $\ovch{\vphi}{-g}{\calH} = \emptyset$, then $\gap_\calH(\vphi) \le g-1$ and thus $\gap(\calH_*) \le g$;
			\item If $\ovch{\vphi}{-g}{\calH} \neq \emptyset$, then $\gap_\calH(\vphi) = g$ and thus $\gap(\calH_*) = g+1$.
		\end{enumerate}
	\end{remark}
}

\pedro{Taking this observation into account, we now propose
Algorithm \ref{a:gc_single_ref}, which finds the least element of $\textgoth{R}_g(\calH, \vphi)$, as shown in Lemma \ref{l:gc_single_ref}.}
%\seba{$\textgoth{R}(\calH, \vphi)$ is the family of refinements of $\calH$ without any constraint.}

% algoritmo aparentemente menos eficiente, con demostracion mas simple
%%%%%%%%%%
\begin{algorithm}[h!]
\caption{Refine one function with gap control}\label{a:gc_single_ref}
\begin{algorithmic}[1]
\Function{GCSingleRefine}{$\calH$,$\vphi$}
\While{ $\exists \vphi' \in \ovch{\vphi}{-g}{\calH}$ }
\State \Call{GCSingleRefine}{$\calH$,$\vphi'$}
\EndWhile
\State \Call{SingleRefine}{$\calH$,$\vphi$}
\EndFunction
\end{algorithmic}
\end{algorithm}

%%%%%%%%%%%%%%%%%%%%%%%%%%%%%%%%%%%%%%%%%%%%%%%%%%%%%%%%%%%%
\begin{lemma}[Properties of GCSingleRefine]\label{l:gc_single_ref}
Let $\calH$ be a hierarchical generator with $\gap(\calH)\leq g$ and
$\vphi\in\calH$.
A call to $\Call{GCSingleRefine}{\calH,\vphi}$
modifies the set \calH, yielding a hierarchical generator $\bar{\calH}\refi 
\calH$,
which is the smallest element of the family $\textgoth{R}_g(\calH, \vphi)$\pedro{, i.e., $\bar \calH$ is the smallest refinement of $\calL\cup\{\vphi\}$ with $\gap$ bounded by $g$}. %\eqref{e:refi_gap_set}.
Furthermore, for each $\psi\in\calR\setminus\{\vphi\}$ there exists $k\in\Z$ with
$1\leq k \leq \lfloor\frac{\l_{\vphi}}{g}\rfloor$ such that
$\psi \in \ovch[k]{\vphi}{-g}{\pedro{\calB}}$. \pedrotodo{Acá decía 
	$\psi \in \ovch[k]{\vphi}{-g}{\calR\cap\calH}$. Esto último es más fuerte y necesita de la versión del Algoritmo~\ref{a:gc_single_ref_v2} pero parece que nunca se usa.}
\end{lemma}

\pedro{
\begin{remark}\label{R:never compute gap}
It is worth noticing that in order to keep the gap bounded by $g$ it is never necessary to compute the gap of a hierarchical generator, which would be rather costly. More precisely, if we start with a hierarchical generator with gap bounded by $g$, such as $\calH = \calB^0$, every hierarchical generator obtained via repeated subsequents calls to \Call{GCSingleRefine}{} will have its gap bounded $g$ automatically. The only thing that must be produced are the sets $O_{\vphi}=\ovch{\vphi}{-g}{\calH}$, which are mere intersections of index sets (see Remark~\ref{R:Osets}).
\end{remark}
}
%%%%%%%%%%%%%%%%%%%%%%%%%%%%%%%%%%%%%%%%%%%%%%%%%%%%%%%%%%%%
\begin{proof}
\pedro{
	Notice first that if $\vphi \in \calB^\l$, then $\ovch{\vphi}{-g}{\calH}$ is a subset of $\calB^{\ell - g}$, so that all the calls to $\Call{GCSingleRefine}{}$ involved in the recursion will be made to B--splines from $\cup_{k=1}^{\lfloor \l/g \rfloor} \calB^{\l-kg} \subset \cup_{j < \l} \calB^j$ which is finite. Hence, the algorithm will end in finite time.
	}
\pedro{
	Moreover, since all the calls to $\Call{GCSingleRefine}{}$ and $\Call{SingleRefine}{}$ will be made with functions from $\calB^j$ with $j < \ell$, after the execution of the while loop, $\vphi$ will still belong to $\calH$ and $\Call{SingleRefine}{\calH,\vphi}$ is a valid call. Since a call to $\Call{SingleRefine}{\calH,\vphi}$ will add to $\calH$ functions of level $\l_\vphi+ 1$ we immediately obtain the last assertion of the Lemma.
}

\pedro{
	We now prove that all the time $\gap(\calH)\le g$. Recall that the first call to $\Call{GCSingleRefine}{}$ is done with $\gap(\calH)\le g$ and notice that $\calH$ is only modified through the execution of line 5 ($\Call{SingleRefine}{}$), from the many calls to the recursive function $\Call{GCSingleRefine}{}$.
	The assertion will be proved if we show that executing line 5 with $\gap(\calH)\le g$ leads to a new hierarchical generator with gap less than or equal to $g$. Notice that line 5 is reached after the while loop has ended, so that $\ovch{\vphi}{-g}{\calH} = \emptyset$, and thus $\gap_\calH(\vphi) \le g-1$. From Lemma~\ref{l:single_ref_gap}, the hierarchical generator obtained after executing line 5 has gap bounded by $g$.
	}
	
\pedro{
	%Notice on the one hand that if $\ovch{\vphi}{-g}{\calH} = \emptyset$ then $\gap_\calH(\vphi) \le g-1$ and the while loop is skipped, moreover, after the call to $\Call{SingleRefine}{\calH,\vphi}$, we have that $\gap(\calH) \le g$. 
	Notice also that if $\ovch{\vphi}{-g}{\calH} \neq\emptyset$, a call to $\Call{SingleRefine}{\calH,\vphi}$ would lead to $\gap(\calH) > g$ according to Remark~\ref{R:necessity-of-refining-ancestors}. It is thus \emph{necessary} to refine all functions in $\ovch{\vphi}{-g}{\calH}$ to maintain $\gap(\calH)\le g$ after executing line 5. This shows that any hierarchical generator in $\textgoth{R}_g(\calH, \vphi)$ must be larger than the one obtained by this algorithm.
}
\end{proof}

\pedro{
\begin{remark}\label{R:gc_single_ref_v2}
It is not difficult to prove that Algorithm~\ref{a:gc_single_ref} is equivalent to Algorithm~\ref{a:gc_single_ref_v2}.
The main difference being that in the latter the set $O_\vphi$ is defined before entering the recursive loop. 
The equivalence relies on the fact that the set $\ovch{\vphi}{-g}{\calH}$ of Algorithm~\ref{a:gc_single_ref} does not increase during the execution of the while loop.
This is easy to see if $g \ge 2$ because in this case, each call to  $\Call{GCSingleRefine}{}$ inside the loop would incorporate into $\calH$ functions of level $\ell_\vphi - k g + 1$, for $k \in \Z_+$, which are never from the same level as those in $\ovch{\vphi}{-g}{\calH}$. The case $g=1$ is also true, but the proof is rather technical.

Observing Algorithm~\ref{a:gc_single_ref_v2} it is easy to conclude that if $\calR = \calH_* \setminus \calH$, with $\calH_*$ the result of $\Call{GCSingleRefine}{\calH,\vphi}$ then for each $\psi\in\calR\setminus\{\vphi\}$ there exists $k\in\Z$ with
 $1\leq k \leq \lfloor\frac{\l_{\vphi}}{g}\rfloor$ such that
 $\psi \in \ovch[k]{\vphi}{-g}{\calR\cap\calH}$. 
\end{remark}
}

% algoritmo mas eficiente, con demostracion mas complicada
%%%%%%%%%%
\begin{algorithm}[h!]
\caption{Refine one function with gap control (version 2)}\label{a:gc_single_ref_v2}
\begin{algorithmic}[1]
\Function{GCSingleRefine}{$\calH$,$\vphi$}
\State $O_{\vphi}=\ovch{\vphi}{-g}{\calH}$
\For{$\vphi'\in O_{\vphi}$}
\State \Call{GCSingleRefine}{$\calH$,$\vphi'$}
\EndFor
\State \Call{SingleRefine}{$\calH$,$\vphi$}
\EndFunction
\end{algorithmic}
\end{algorithm}
%%%%%

\begin{comment}
%%%%%%%%%%%%
\begin{definition}[Distance between functions]
 Given $\vphi$ and $\vphi'$ in $\calH$ we define
 $\dist(\vphi,\vphi')=\dist(w_\vphi, w_{\vphi'}) = \inf\{\dist(x,y): x\in 
w_\vphi, y\in w_{\vphi'}\}$.
\end{definition}
%%%%%%%%
\end{comment}
%%seba

Let $\calH$ be a hierarchical generator with $\fgap[]{\calH} \le g$ and let
$\calM\subset\calH$ be a given set of functions to be refined, 
Algorithm~\ref{a:refine_gc} will refine the functions maintaining the gap under control 
($\le g$).
%%%%
\begin{algorithm}[h]
\caption{Refine $\calM$ with gap control}\label{a:refine_gc}
\begin{algorithmic}[1]
\Function{GCRefine}{$\calH$,$\calM$}
\While{$\pedro{\exists} \vphi \in \calM\cap\calH$}
\State \Call{GCSingleRefine}{$\calH$,$\vphi$}
\EndWhile
\EndFunction
\end{algorithmic}
\end{algorithm}
%%%%

\begin{lemma}[Properties of \Call{GCRefine}{}]\label{l:gc_refine}
 Let $\calH$ be a hierarchical generator with $\gap(\calH)\leq g$ and
$\calM\subset\calH$.
A call to  \Call{GCRefine}{$\calH$,$\calM$} of Algorithm \ref{a:refine_gc} 
modifies the set \calH, yielding a hierarchical generator $\bar{\calH}\refi 
\calH$, and a refiner set $\bar{\calR}:=\bar{\calL}\setminus\calL$ satisfying:
\begin{enumerate}[(i)]
 \item $\calM\subset\bar\calR$
 \item $\gap{\bar\calH} \leq g$
 \item 
 $\bar\calR \setminus \calM 
\subset\cup_{\vphi\in\calM}\cup_{k=1}^{\lfloor\frac{\l_{\vphi}}{g}\rfloor} 
\ovch[k]{\vphi}{-g}{\calB}$
\item\label{l:gc_refine:R} for each $\psi\in\bar\calR\setminus \calM$ 
there is $\vphi\in\calM$ and $k > 0$ such that
$\psi\in\ovch[k]{\vphi}{-g}{\calB}$
\item\label{l:gc_refine:new}
if $\eta\in\bar\calH\setminus\calH$ then
there exists $\vphi \in \calM$ such that either $\eta\in\ch{\vphi}$ or $ 
\eta\in\ch{\ovch[k]{\vphi}{-g}{\calB}}$ for some $k > 0$.
 \end{enumerate}
\end{lemma}

% \begin{lemma}[Properties of \Call{RefineGC}{}]\label{l:refine_gc}
% Let $\calH=\calH(\calL)$ be a hierarchical generator with $\gap(\calH)\leq g$
% and $\calM\subset\calH$.
% A call to the \Call{RefineGC}{$\calH$,$\calM$} of Algorithm \ref{a:refine_gc} 
% finishes giving a hierarchical generator $\bar\calH$ such that:
% \begin{enumerate}[(i)]
%  \item $\bar\calL = \calL \dcup \calM \dcup \bar\calM$
%  \item $\gap{\bar\calH} \leq g$
%  \item If $\vphi_0\in\bar\calH\setminus\calH$ then
%  there exists a constant $C$ only depending on $m$ and $d$
%  and $\vphi\in\calM^{\l_{\vphi_0}+jg-1}$ for some  such that
%  $\dist(\vphi_0, \vphi) \leq C 2^{-\l_{\vphi_0}}$.
% \end{enumerate}
% \end{lemma}
\begin{proof}
\pedro{Assertions (i) and (ii) are immediate consequences of Lemma~\ref{l:gc_single_ref} and assertion (iii) follows from Remark~\ref{R:gc_single_ref_v2} y (iv)--(v) son consecuencias directas de (i)--(iii).}
 Observe that if $\calM\not=\emptyset$ the call to \Call{GCSingleRefine}{} 
inside the {\bf while} loop in Algorithm \ref{a:refine_gc} is
 executed at least one time and at most $\#\calM$ times.
 Let $\vphi_1,\dots,\vphi_M$ be the functions in $\calM$ that were passed to 
 \Call{GCSingleRefine}{} in sequencial order inside the loop.
 Let $\calH_0=\calH$, and $\calH_1,\dots,\calH_M$ be the hierarchical generators 
obtained after each 
iteration of the
 {\bf while} loop, then $\calH_M=\bar\calH$ and $M \le \#\calM$.
From Lemma \ref{l:gc_single_ref} it follows that each $\calH_j$ is a 
hierarchical generator
and $\gap{\calH_j}\leq g$, so the same holds for $\calH$.
Moreover, $\calR=\calR_1\dcup\dots\dcup\calR_{M-1}$, with $\calR_j = \calL_j 
\setminus \calL_{j-1}$. Then, if $\psi\in\bar\calR$ 
it must belong to one of the $\calR_j$ and again Lemma \ref{l:gc_single_ref} 
implies (iii), which immediately implies (iv) and (v).
% 
% then this function must have appeared in
% some $\calH_j$, more precisely, there is an integer $\hat j$ with $1\leq \hat 
%j\leq J$
% such that $\vphi_0\in\calH_{\hat j}\setminus\calH_{\hat j - 1}$ thus 
% Lemma \ref{l:refine_gc} implies that $\vphi_{\hat 
%j}\in\calM^{\l_{\vphi_0}-jg+1}$
% and $\dist(\vphi_0, \vphi_{\hat j}) \leq C 2^{-\l_{\vphi_0}}$.
% Letting $\vphi=\vphi_{\hat j}$ the result follows.
 \end{proof}

\begin{algorithm}[h]
\caption{Refine $\calM$ with gap and absorbing constraints}\label{a:refine}
\begin{algorithmic}
\Function{GARefine}{$\calH$,$\calM$}
\State \Call{GCRefine}{$\calH$,$\calM$}
\State $\tilde\calH$ = \Call{AbsRefine}{$\calH$}
\State \Return $\tilde\calH$
\EndFunction
\end{algorithmic}
\end{algorithm}

The following result shows that if a function is refined in
the process to make the generator absorbing then its cause can be traced 
to a function refined in the first step of refinement, i.e. in the gap controlled refinement step.
\begin{lemma}[Absorbing refinement to gap control]\label{l:max_to_gc}
Let $\calH$ be an absorbing basis with $\gap(\calH)\leq g$
and $\calM\subset\calH$.
Let $\bar\calH$ be the state of the generator \calH right after the call to  
\Call{GCRefine}{$\calH$,$\calM$} in Algorithm \ref{a:refine}.
Let $\bar\calR \pedro{= \bar{\calL}\setminus \calL}$ be the refiner of 
$\bar\calH$ with respect to \calH
and $\tilde\calR = \tilde{\calL}\setminus \bar{\calL}$ the refiner of 
$\tilde\calH$ with respect to 
$\bar\calH$.
Then, there exists a constant $C$, only depending on $m$, $n$ and $g$, such that
for each $\psi\in\tilde\calR$ there is $\vphi\in\calM$ with
$|\rho(\vphi,\psi )| \leq C$ and $\l_\vphi - \l_\psi \ge -g$.
\end{lemma}

\begin{proof}
Let $\psi\in\tilde\calR$ then by Lemma 
\ref{l:ref_cause}\ref{l:ref_cause:1}
there is $\eta\in\bar\calH\setminus\tilde\calH$ such that
$\psi\in\ch[k_0]{\eta}$ for some $k_0\geq 0$.
Using Lemma~\ref{l:abs_ref}\ref{e:gap1} and
Lemma~\ref{p:gap_ch} there is $k_1>0$ such that
$\emptyset \not= \ovch{\psi}{k_1}{\bar{\calH}}
\subset  \ovch{\eta}{k_1+k_0}{\bar{\calH}}$.
From where we get
$g\ge\gap{\calH}=\gap{\bar{\calH}}\geq k_1+k_0 > k_0$.
Now we consider the two possible cases:

\begin{case} ($\bbI(\eta)\subset \bbI(\calL^{\l_{\eta}})$). 
%$\eta\ab\calL^{\l_{\eta}}$).
If $\bbI(\eta)\subset \bbI(\calL^{\l_{\eta}})$, 
%$\eta\ab\calL^{\l_{\eta}}$ 
then $\eta \notin \calH$ because $\calH$ is absorbing, and thus 
$\eta\in\bar\calH\setminus\calH$.
Now from Lemma~\ref{l:gc_refine}\ref{l:gc_refine:new}
there exist $\vphi\in\calM$ and  $k_2\geq 0$ such that 
$\eta\in\ch{\ovch[k_2]{\vphi}{-g}{\calB}}$, thus
$\psi\in\ch[k_0+1]{\ovch[k_2]{\vphi}{-g}{\calB}}$
and $\l_\psi - \l_\vphi = k_0 + 1 - k_2 g \le g$.
Using that $\l_\psi-\l_\eta=k_0$ and Lemma~\ref{l:rho_triangle}  we get
\begin{equation}
\rho(\vphi,\psi) = n^{k_0} \rho(\vphi,\eta) + \rho(\eta,\psi).
\end{equation}
On the one hand, since there exists $\xi\in\ovch[k_2]{\vphi}{-g}{\calB}$ which 
is a parent of $\eta$, 
Lemma~\ref{l:rho_triangle} yields
$\rho(\vphi,\eta)= n\rho(\vphi,\xi)+\rho(\xi,\eta)$.
By Lemma~\ref{l:desc_rho}, $|\rho(\xi,\eta)|\leq m(n-1)$,
and from Lemma~\ref{l:overlapping}~\ref{l:overlapping:rho},
$|\rho(\vphi,\xi)|\leq C\le 2m$ so that
\begin{equation*}
|\rho(\vphi,\eta)|\leq n(C+m)\le 3nm.
\end{equation*}
On the other hand, since $\psi\in\ch[k_0]{\eta}$, Lemma~\ref{l:desc_rho} implies
\begin{equation}\label{e:eta_psi}
 |\rho(\eta,\psi)| \leq n^{k_0} m,
\end{equation}
from where 
\begin{equation}\label{caso1}
|\rho(\vphi,\psi)|\leq  n^{g} (3n+1) m,
\end{equation}
because $k_0 \le g$.
\end{case}

\begin{case} ($\bbI(\eta)\not\subset \bbI(\calL^{\l_{\eta}})$). 
%$\eta\not\ab\calL^{\l_{\eta}}$).
	If $\bbI(\eta)\not\subset \bbI(\calL^{\l_{\eta}})$, 
	%since $\calH$ is absorbing, 
	$\eta \notin \calL$ and also $\eta \in \bar \calH \setminus 
\tilde{\calH}$ so that $\eta \notin \tilde{\calH}$ and thus 
$\eta\in\tilde{\calR}$.
	By Lemma~\ref{l:abs_ref}\ref{l:abs_ref:absd} and the fact that 
$\bbI(\eta)\not\subset \bbI(\calL^{\l_{\eta}})$, it follows that there is 
$\mu\in \big(\bar\calL^{\l_\eta}\setminus \calL^{\l_\eta}\big) 
\cap\ovch[]{\eta}{0}{\calB} = 
\bar\calR^{\l_{\eta}}\cap\ovch[]{\eta}{0}{\calB}$
%\marcetodo{$\big(\calL^{\l_\eta}\setminus \calL^{\l_\eta}\big)$ vacio?}.
%%
Since $\mu\in\bar\calR$, Lemma~\ref{l:gc_refine}\ref{l:gc_refine:R} implies
there is $\vphi\in\calM$ and $k_2\geq 0$ such that
$\mu\in\ovch[k_2]{\vphi}{-g}{\calB}$. Thus using
Lemma~\ref{l:rho_triangle} and that 
$\l_\psi-\l_\eta=\l_\psi-\l_\mu=k_0$ we get 
$\l_\vphi - \l_\vphi = \l_\vphi - \l_\eta + \l_\eta - \l_\psi = k_2g - 
k_0 \ge -g$ and
\begin{equation}
\rho(\vphi,\psi) = n^{k_0} \rho(\vphi,\mu) + 
 n^{k_0} \rho(\mu,\eta) +\rho(\eta,\psi).
\end{equation}
Lemma~\ref{l:overlapping}\ref{l:overlapping:rho} leads to
$|\rho(\vphi,\mu)|\leq 2m$ and Remark~\ref{R:ballofsplinessamelevel} yields $|\rho(\mu,\eta)|\leq m$, whence
\begin{equation}\label{caso2}
|\rho(\vphi,\psi)| \leq 4mn^{k_0} \le 4mn^g
\end{equation}
due to  \eqref{e:eta_psi} and the fact that $k_0 \le g$.
\end{case}
The assertion follows from~\eqref{caso1} and~\eqref{caso2}.
%\todo[inline]{all the constants in this proof can be improved using m, g and n}
\end{proof}

\begin{theorem}[Properties of \Call{GARefine}{}]\label{l:refine}
 Let $\calH$ be an absorbing hierarchical basis with $\gap{\calH}\leq g$
 and $\calM\subset\calH$.
 Let $\calH_*$ be the output of \Call{GARefine}{$\calH$,$\calM$}
 in Algorithm~\ref{a:refine}.
 Then $\calH_* \refi \calH$ is an absorbing hierarchical basis with 
$\calM \subset \calR = \calL_* \setminus \calL$, $\gap{\calH_*}\leq g$ and 
moreover, 
 for each $\vphi_*\in\calH_*\setminus\calH$ there exists $\vphi\in\calM$ such 
that
$|\rho(\vphi, \vphi_*)| \leq C$ and $\l_\vphi-\l_{\vphi_*}\geq -g$, 
with a constant $C$ only depending on $m,n$ and $g$.
\end{theorem}

\begin{proof}
Let $\bar\calH$ be the state of the generator \calH right after the call to  
\Call{GCRefine}{$\calH$,$\calM$} in Algorithm \ref{a:refine}.
From Lemma~\ref{l:gc_refine}, $\bar\calH$ is a hierarchical generator with 
$\gap{\bar\calH}\leq g$.
Now $\calH_*$ is the output of \Call{AbsRefine}{$\pedro{\bar\calH}$}, hence
Lemma~\ref{l:abs_ref} and Proposition~\ref{p:max_gap} imply that $\calH_*$ is an 
absorbing basis with $\gap{\calH_*}\leq g$.

Let $\vphi_*\in\calH_*\setminus\calH$, then from 
Lemma~\ref{l:ref_cause}\ref{l:ref_cause:2} there is
$\psi\in\calR$
such that $\vphi_*\in\ch{\psi}$. 

If $\psi \in \calM$, then $\vphi = \psi$ 
satisfies the assertion because $\l_\psi - \l_{\vphi_*}= -1 \ge -g$ and $|\rho(\psi,\vphi_*) | 
\le mn$ due to Lemma~\ref{l:desc_rho}.

If $\psi \in \calR\setminus\calM$, then either $\psi \in 
\bar{\calR}\setminus\calM$ or $\psi \in \tilde{\calR}\setminus\calM$,
%Clearly $\calR= \bar\calR \cup \tilde\calR$, 
with $\bar{\calR} = \bar{\calL}\setminus \calL$ and $\tilde{\calR} = \calL_* 
\setminus \bar{\calL}$.
%, Lemma~\ref{l:gc_refine}\ref{l:gc_refine:R} and Lemma~\ref{l:max_to_gc} 
% imply that $|\rho(\vphi, \vphi_*)| \leq C$.
If $\psi\in\bar\calR$ then Lemma~\ref{l:gc_refine}\ref{l:gc_refine:R} implies 
the existence of $\vphi \in \calM$ such that
$\psi\in\ovch[k]{\vphi}{-g}{\calB}$, for some $k > 0$, so that
$\l_\vphi \ge \l_\psi \pedro{+} g = \l_{\vphi_*}-1\pedro{+}g$ and
\[
|\rho(\vphi,\vphi_*)| =  |\rho (\psi,\vphi_*) + n \rho(\vphi,\psi) | \le C
\text{ and }
\l_\vphi -\l_{\vphi_*} \ge \pedro{0 \ge}-g,
\]
due to Lemma~\ref{l:rho_triangle}.
If $\vphi\in\tilde\calR\setminus\calM$ then
Lemma~\ref{l:max_to_gc} implies that 
$|\rho(\vphi, \vphi_*)| \leq C$
and $\l_\vphi-\l_{\vphi_*}\ge -g$.
\end{proof}

\begin{remark}
	Following the steps of this proof and those of Lemma~\ref{l:max_to_gc}
	it can be easily seen that the alluded constant $C$ is bounded by $4m^g$.
\end{remark}

\section{Complexity of Refinement}\label{s:complexity}

We consider a sequence of refinements $\{\calH_r\}$ generated
by subsequent calls of the form
\[
\calH_{r+1} = \Call{GARefine}{\calH_r,\calM_r}, 
\qquad\text{with } \calM_r \subset \calH_r,
\]
as the one obtained by a typical adaptive loop,
which we denote as
%%%
%\begin{algorithmic}
%\For{$r\geq 0}$
%\State determine a suitable subset $\calM_r\subset\calH_r$
%\State $\calH_{r+1}$ = \Call{GARefine}{$\calH_r$,$\calM_r$}
%\EndFor
%\end{algorithmic}
%%%
%This is to say we have a sequence of successive refinements of the form
\begin{equation}
 \calH_0 \xrightarrow{\calM_0} \calH_1 \xrightarrow{\calM_1} 
\quad\dots\quad 
\xrightarrow{\calM_{R-1}} \calH_R.
\end{equation}
We assume, for simplicity that $\calH_0=\calB^0$.
Let $D=CB$, where $C$ is the constant in Theorem~\ref{l:refine} and
B is given below.
And define the ``reach'' of a function $\vphi\in\calB$ as
$\calN(\vphi):= \left\{\psi \in \calB : \l_\psi \le \l_\vphi+g \text{ and } 
|\rho(\vphi,\psi) | \le D \right\}$, 
which is equivalent to defining it as $\calN(\vphi) = \cup_{k\geq -g} 
B(\vphi,D, -k)$, after recalling
from Definition~\ref{d:ballofsplines} that
$\ball{\vphi}{D}{-k} :=
\{\psi\in\calB^{\l_\vphi-k} :
|\rho(\vphi,\psi)|\leq D \}$.
Let $\calH_* = \calH_R$ and $\calM = \cup_{r=0}^{R-1} \calM_r$
and consider the following allocation function 
$\lambda:\calM\times\calH_*\to\R$ given by
\begin{equation}\label{lambda}
 \lambda(\vphi,\vphi_*)=
 \begin{cases}
  a(\l_{\vphi}-\l_{\vphi_*}) &\text{if } \vphi_*\in\mathcal{N}(\vphi)\\
  0 &\text{otherwise}.\\
 \end{cases}
\end{equation}
Where $a(k)$ is a decreasing sequence such that
$\sum_{k=-g}^{\infty} a(k) = A < \infty$ and there
is another increasing sequence $b(k)$ with $b(0)\ge 1$,
$\sum_{k=-g}^{\infty} b(k)n^{-k} = B < \infty$ and
$\inf_{k\ge-g} a(k)b(k)=:c_*>0$.
For example consider $a(k)=(k+(g+1))^{-2}$ and $b(k)=n^{k/2}$, which satisfy 
these assumptions.
% so that
%$B=\frac{\sqrt{n}}{\sqrt{n}-1}$ \pedro{ojo, no es cierto porque se suma desde 
% $-g-1$} and $c_*=(\log n)^2 n^{2/\log n -g-1...}$ and $A=\pi^2/6 %\approx 1.64$.

\begin{lemma}[Upper bound]
	For any $\vphi \in \calM := \cup_{r=0}^{R-1} \calM_r$,
% There exists a constant $C_1 = C^d A$ such that
 \[
 \sum_{\vphi_*\in\calH_*\setminus\calH_0} \lambda(\vphi,\vphi_*) \leq C_U := 
(2D+1)^d A.
 \]
\end{lemma}
\begin{proof}
 Let $\vphi\in\calM$ then, due to~\eqref{lambda} and the definition of 
$\calN(\vphi)$ 
 \begin{align*}
 	 \sum_{\vphi_*\in\calH_*\setminus\calH_0} \lambda(\vphi,\vphi_*) &\leq
 	\sum_{\vphi_*\in\calN(\vphi)} \lambda(\vphi,\vphi_*)
 	\\
 	&= 	\sum_{k=-g-1}^{\infty} a(k) \, \#B(\vphi,D, -k)
 	\\
 	&\leq (2D+1)^d \sum_{k=-g-1}^{\infty} a(k) = (2D+1)^d A.
 \end{align*}
where in the last inequality we have used Lemma~\ref{l:radius}
\end{proof}

\begin{lemma}[Lower bound]
	For every $\vphi \in \calH_* \setminus \calH_0$,
\[
\sum_{\vphi\in\calM} \lambda(\vphi,\vphi_*) \geq C_L := \inf_{k \ge -g} b(k) a(k),
\qquad\text{with } \calM = \cup_{r=0}^{R-1} \calM_r.
\] 
\end{lemma}
\begin{proof}
Let $\vphi_0 = \vphi_* \in \calH_*\setminus\calH_0$, there must exist an integer 
$r_0$ 
with $0<r_0\leq R$ such
that $\vphi_0 \in \calH_{r_0} \setminus \calH_{r_0-1}$.
Thus from Theorem~\ref{l:refine} there is 
$\vphi_1\in\calM_{r_0-1}$ such that $|\rho(\vphi_1,\vphi_0)|\le C\leq D$
and $\l_{\vphi_1}-\l_{\vphi_0}\geq -g$, thus $ \vphi_0\in 
\mathcal{N}(\vphi_1)$.
If now $\vphi_1 \in \calH_{r_0-1}\setminus\calH_0$ we can repeat the process
and find $r_1$ with $r_1<r_0\leq R$ such that
$\vphi_1 \in \calH_{r_1} \setminus \calH_{r_1-1}$ and
from Theorem~\ref{l:refine} there is 
$\vphi_2\in\calM_{r_1-1}$ such that 
$|\rho(\vphi_2,\vphi_1)|\le C\leq D$
and $\l_{\vphi_2}-\l_{\vphi_1}\geq -g$, 
whence $\vphi_1 \in \mathcal{N}(\vphi_2)$.

This process can be repeated to find a sequence
of B--splines
$\vphi_0,\dots,\vphi_J$ with $J\geq 1$ and
integers $R\geq r_0>r_1>\dots>r_J=1$ such that:
\begin{itemize}
\item $\vphi_{j+1}\in\calM_{r_j-1}\setminus\calH_0$ for $j\in[0:J-2]$;
\item $\vphi_J\in\calM_{0} \subset \calH_0$;
\item $|\rho(\vphi_{j+1},\vphi_j)| \le C$, and $\l_{\vphi_{j+1}}-\l_{\vphi_j} 
\ge -g$, for $j\in[0:J-1]$;
\item $\vphi_j \in \mathcal{N}(\vphi_{j+1})$, for $j\in[0:J-1]$.
\end{itemize}

It is worth observing that $\vphi_J\in \calH_0$, so that  
$\l_{\vphi_J}=0$ (and $J\geq 1$), % and $|\rho(\vphi_{j+1},\vphi_j)|<C$ 
and 
$\l_{\vphi_{j+1}}-\l_{\vphi_{j}}\in[-g,\infty)$, thus,
the level of the $\vphi_j$'s, as $j$ increases in the sequence, can increase in 
any integer amount, but when it decreases, it will do so in steps smaller than 
$g+1$. 
Since $\l_{\vphi_J}=0$, there exists an integer $s$ with
$0< s \leq J$ such that 
$\l_{\vphi_s} < \l_{\vphi_0}$ and
$\l_{\vphi_0} \leq \l_{\vphi_j}$ for all $j=0,\dots,s-1$.

Let $k_i=\l_{\vphi_{i}}-\l_{\vphi_{0}}$, using Lemma \ref{l:rho_triangle}, for 
any 
$j=1,\dots,s$ we have
\begin{equation*}
  \rho(\vphi_j, \vphi_0) =  
  \sum_{i=0}^{j-1}\frac{1}{n^{k_i}} \rho(\vphi_{i+1},\vphi_i),
  \quad\text{so that}\quad |\rho(\vphi_j,\vphi_0)|<C 
\sum_{i=0}^{j-1}\frac{1}{n^{k_i}}.
\end{equation*}

Consider for any $k\geq 0$ and $1\leq j \leq s$ the set
\[
\calB(k,j)=\{\vphi\in\{\vphi_0,\dots,\vphi_{j-1}\} : \l_\vphi = 
\l_{\vphi_0}+k \},
\quad\text{and let}\quad
m(k,j)=\#\calB(k,j).
\]
Clearly for a fixed $k$ the function $m(k,j)$ is increasing with $j$.
Then rearranging the terms of the sum we have
$\sum_{i=1}^{j-1}\frac{1}{n^{k_i}} = 
 \sum_{k=0}^{\infty}\frac{1}{n^{k}} m(k,j)$.

Let $K=\{k\in\Z^+_0 :  m(k,j)>b(k) \text{ for some } j\in[0:s]\}$,
and consider two possible cases.

\begin{case} ($K=\emptyset$)
In this case $m(k,s)\leq b(k)$ for all $k$ so we obtain that, for $1\le j\le s$,
\[
\sum_{i=0}^{j-1}\frac{1}{n^{k_i}} = 
\sum_{k=0}^{\infty}\frac{1}{n^{k}} m(k,j)
\le \sum_{k=0}^{\infty}\frac{1}{n^{k}} b(k) = B,
\quad\text{hence}\quad
|\rho(\vphi_j,\vphi_0)| \leq C B = D,
\]
% $\sum_{i=1}^{j-1}\frac{1}{n^{k_i}} \leq \sum_{k=0}^{\infty}\frac{1}{n^{k}} 
% b(k) = B$.
%Hence
%$|\rho(\vphi_j,\vphi_0)|<C \sum_{i=1}^{j-1}\frac{1}{n^{k_i}}\leq C B$,
so that $\vphi_0 = \vphi_* \in \calN(\vphi_s)$ and $\lambda(\vphi_s,\vphi_*) = 
a(\l_{\vphi_s}-\l_{\vphi_*}) \ge a(0) > 0$ because $a(\cdot)$ is decreasing and 
$\l_{\vphi_s}<\l_{\vphi_*}$. Therefore, since $\vphi_s \in \calM$,
$\sum_{\vphi\in\calM} \lambda(\vphi,\vphi_*) \geq \lambda(\vphi_s, 
\vphi_0)\geq a(0) > 0$.
\end{case}

\begin{case} 
($K\not=\emptyset$):
For each $k\in K$, let $j(k) = \min\{ j \in\{1,\dots,s\} : m(k,j) > b(k) \}$, so 
that $j(k) \ge 2$ because $m(k,1) \le 1 = b(0) \le b(k)$ for any $k \in \Z_0^+$, 
and 
$m(k, j')\leq b(k)$ if $0\leq j' < j(k)$. 
Let now $\hat j = \min\{ j(k): k \in K\}$ and $\hat k$ the minimum $k$ that 
verifies $\hat j = j(k)$, 
hence
$m(k,j) \le b(k)$ for all $k \in \Z_0^+$ if $1\le j < \hat j$.
Then
\[
\sum_{\vphi\in\calM} \lambda(\vphi,\vphi_*)
\ge 
\sum_{\vphi \in \calB(\hat k, \hat j)} \lambda(\vphi,\vphi_*)
\]
Let now $\vphi \in \calB(\hat k, \hat j)$ and compute $\lambda(\vphi,\vphi_*)$.
The first step is to determine that $\vphi_* \in \calN(\vphi)$.
On the one hand, by definition of $\calB(\hat k, \hat j)$, $\l_{\vphi_*} = 
\l_\vphi - \hat k \le \l_\vphi \le \l_\vphi + g $.
On the other hand, there exists $j < \hat j$ such that $\vphi = \vphi_j$ so 
that, as before
\[
|\rho(\vphi,\vphi_*)| = |\rho(\vphi_j,\vphi_*)| \le C \sum_{i=0}^{j-1} 
\frac{1}{n^{k_i}} 
= C \sum_{k=0}^\infty \frac{1}{n^k} m(k,j) 
\le C \sum_{k=0}^\infty \frac{1}{n^k} b(k) = CB = D,
\]
and thus $\vphi_* \in \calN(\vphi)$.
Therefore, $\lambda(\vphi,\vphi_*) = a(\l_{\vphi} - \l_{\vphi_*}) = a(\hat k)$ 
for each $\vphi \in \calB(\hat k, \hat j)$ and thus
\[
\sum_{\vphi\in\calM} \lambda(\vphi,\vphi_*)
%\ge 
%\sum_{\vphi \in \calB(\hat k, \hat j)} \lambda(\vphi,\vphi_*)
\ge \#\calB(\hat k, \hat j) a(\hat k) \ge b(\hat k) a(\hat k) \ge
\inf_{k \ge -g} b(k) a(k) = C_L > 0. 
\]
The assertion thus follows.
\end{case}
\end{proof}

We are now in position to prove the following complexity estimate.

\begin{theorem}
Assume that the sequence of hierarchical gap-controlled absorbing bases $\{\calH_r\}$ has been generated
by subsequent calls of the form
\[
\calH_{r+1} = \Call{GARefine}{\calH_r,\calM_r}, 
\qquad\text{with } \calM_r \subset \calH_r, \qquad r=0,1,\dots,
\]
with $\calH_0 = \calB^0$. Then
\[
\# \calH_R - \#\calH_0 \le \frac{C_U}{C_L} \sum_{r=0}^{R-1} \#\calM_r
\]
for any $R$.
\end{theorem}

\begin{proof}
By the previous two lemmas, letting $\calM = \cup_{r=0}^{R-1} \calM_r$ we immediately obtain
\begin{align*}
\#\calH_R -\#\calH_0 &\le \frac{1}{C_L} \sum_{\vphi_* \in \calH_R\setminus\calH_0} C_L
\le \frac{1}{C_L} \sum_{\vphi_* \in \calH_R\setminus\calH_0} \sum_{\vphi \in \calM } \lambda(\vphi,\vphi_*)
\\
&= \frac{1}{C_L} \sum_{\vphi \in \calM } \sum_{\vphi_* \in \calH_R\setminus\calH_0} \lambda(\vphi,\vphi_*)
\le \frac{C_U}{C_L} \# \calM.
\end{align*}
\end{proof}

We finally make an interesting observation about approximation classes.

Given $s> 0$ and a function space $\bbV$ over $\Omega$ with norm $\| \cdot \|$ we define the best approximation error with complexity $N$ as
\[
\sigma_N(u) = \inf_{\#\calH-\#\calH_0\le N}\, \inf_{V \in \spn{\calH}} \| u - V \|, \qquad u \in \bbV.
\]
and the approximation class $\A_s$ as
\[
\A_s = \big\{ v \in \bbV : \sigma_N(u) \le C N^{-s}, N \in \mathbb{N} \big\}.
\]
We have two definitions, depending if we consider any hierarchical space (generated by hierarchical generators) or absorbing and gap controlled hierarchical spaces, i.e.,
\begin{description}
\item[$\sigma_N$, $\A_s$:] considering absorbing and gap controlled hierarchical spaces (fixed $g > 0$);
\item[$\overline\sigma_N$, $\overline\A_s$:] considering all hierarchical spaces.
\end{description}

Clearly, $\overline\sigma_N(u) \le \sigma_N(u)$, so that
\[
\A_s \subset \overline\A_s.
\]
But also
\[
\overline\A_s \subset \A_s,
\]
i.e. if a function can be approximated with hierarchical spaces at a rate $N^{-s}$ it can also be approximated at the same rate with absorbing and gap controlled hierarchical spaces.

This is an immediate consequence of the following Proposition.

\begin{proposition}
For each hierarchical generator $\overline\calH$ there exists an absorbing gap controlled hierarchical basis $\calH$ with
\[
\spn{\overline\calH}\subset \spn\calH
\quad\text{and}\quad 
\#\calH - \#\calH_0 \Cle \#\overline\calH - \#\calH_0
\]
\end{proposition}

\begin{proof}
Given a hierarchical generator $\overline\calH$ construct an absorbing gap controlled hierarchical basis as follows:

\begin{algorithmic}[1]
\State $\calH = \calB^0$
\For{$\l=0,1,\dots$} % to $\depth(\calH)-1$}%\Comment{$N$ is the depth of \calH}
\State $\calM_\ell = \overline\calL^\ell \cap \calH$
\State $\Call{GARefine}{\calH,\calM_\ell}$
\EndFor
\end{algorithmic} 

Then  $\calL \supset \overline\calL$, $\calH \prec \overline\calH$
and
\[
\#\calH - \#\calH_0 \Cle \sum_{\ell=0}^L \#\calM_\ell
\le \# \overline\calL \Cle \#\overline\calH - \#\calH_0.
\]
\end{proof}

\section{Appendix}\label{sec:appendix}
In this section we present some auxiliary results and proofs which are simple, but a little bit technical, 
and would have obstructed the reading of the previous sections where they were presented there.

Let $n$ be a fixed integer such that $n>1$, and let $\operatorname{D}_m$ and  
$\operatorname{M}_m$ be the index functions defined in 
Section~\ref{subsec:index_functions}. For any $k \in \Z^+$ we have the following 
results.
\begin{lemma}[Formulas I]\label{l:form}
	\begin{enumerate}[(i)]
		\item $\ipro[k]{m}{\vec i} = n^k \vec i  + m \frac{n^k-1}{n-1}$.
		\item $\fdiv[k]{m}{\vec i} = \frac{\vec i}{n^k} -
		\frac{m}{n^k}\frac{n^k-1}{n-1}
		- \vec R$, with $\vec R\in\left[0,1-\frac{1}{n^k}\right]^d$.
	\end{enumerate}
\end{lemma}
\begin{proof} Note that for $k=1$, both (i) and (ii), are just the definition of 
$\operatorname{M}_m$ and $\operatorname{D}_m$. The case $k>1$ follows 
immediately for~(i).
	In order to prove (ii) let $\vec r_j$ be such that $\fdiv[j]{m}{i}= 
\frac{\fdiv[j-1]{m}{i}-m}{n}+ \vec r_j$, for  $j=1,\ldots,k$. Then $0\le \vec 
r_j \le 1-1/n$ and $\vec R=\frac{1}{n}\sum^{k-1}_{j=0}\frac{\vec r_{k-j}}{n^j}$. 
Therefore $0 \le \vec R \le 1-1/n^k$.
\end{proof}

\begin{lemma}[Formulas II]\label{lem:L_R_properties} 
	For any $k \in \Z^+$ we have that
	\begin{enumerate}[(i)]
		\item $\operatorname{L}^k(\vec i) = 
		\frac{\vec i}{n^{k g}} 
		- \frac{p}{n^{(k-1)g}} \frac{n^{kg}-1}{n^g-1} - A$,
		with $0\leq A \leq1-\frac{1}{n^{k g}}$. 
		\item $\operatorname{R}^k(\vec i) = 
		\frac{\vec i}{n^{k g}} 
		+ \frac{p}{n^{kg}} \frac{n^{kg}-1}{n^g-1} - B$,
		with $0\leq B \leq1-\frac{1}{n^{k g}}$.   
	\end{enumerate}
\end{lemma}
\begin{proof}
	The result follows by induction applying Lemma~\ref{l:form}~(ii).
\end{proof}

\begin{lemma}[Inverse]\label{l:inv1}
	Let $p,q\in\Z$ then
	\begin{enumerate}[(i)]
		\item if $0\leq p-q < n$ then
		$\fdiv[k]{q}{\ipro[k]{p}{\vec i}} = \vec i$.
		\item if $m\in\Z$ then
		$\vec i - (n^k-1) \leq \ipro[k]{m}{\fdiv[k]{m}{\vec i}} 
		\leq \vec i$.
	\end{enumerate}
\end{lemma}
%%%%%%
\begin{proof}
	$\fdiv{q}{\ipro{p}{\vec i}} =\lfloor \vec i + \frac{p-q}{n} \rfloor$
	which equals $\vec{i}$ if $0\leq p-q < n$.
	For $k>1$ it is a matter of iterating the previous result.
	
	To show (ii) observe that applying Lemma~\ref{l:form}~(ii) and then 
Lemma~\ref{l:form}~(i) we have that $\ipro[k]{m}{\fdiv[k]{m}{\vec i}} = i - 
n^kR$, for some $R$ such that $0\leq R \leq 1 - \frac{1}{n^k}$. So the result is 
clear.
\end{proof}

\begin{lemma}[Inequalities]\label{l:inequalities}  
	\begin{enumerate}[(i)]
		\item $\ipro[k]{m}{\vec j} \leq \vec i$ if and only if
		$\vec j \leq \fdiv[k]{m}{\vec i}$.
		\item $\ipro[k]{m}{\vec j} \geq \vec i$ if and only if
		$\vec j \geq \fdiv[k]{m-(n-1)}{\vec i}$.
	\end{enumerate}
\end{lemma}
\begin{proof}
	First note that $\operatorname{M}_m^k$ and $\operatorname{D}_m^k$ are non 
decreasing index functions. Hence if $\ipro[k]{m}{\vec j} \leq \vec i$ then 
$\fdiv[k]{m}{\ipro[k]{m}{\vec j}} \leq \fdiv[k]{m}{\vec i}$ and by 
Lemma~\ref{l:inv1}~(i) we have that $\vec j \leq \fdiv[k]{m}{\vec i}$. The other 
implication follows exactly in the same way applying Lemma~\ref{l:inv1}~(ii).
	
	Furthermore, if  $\ipro[k]{m}{\vec j} \geq \vec i$ taking  $p=m$ and 
$q=m-(n-1)$ in Lemma~\ref{l:inv1}~(i) we have  $\vec j \geq 
\fdiv[k]{m-(n-1)}{\vec i}$. On the other hand, by Lemma~\ref{l:form}~(ii), we 
can show that $\ipro[k]{m}{\fdiv[k]{m - (n-1)}{\vec i}} = i + n^kR$, for some 
$R$ such that $0\leq R \leq 1 - \frac{1}{n^k}$. Hence if
	$\vec j \geq \fdiv[k]{m-(n-1)}{\vec i}$ then $\ipro[k]{m}{\vec j} \geq 
\ipro[k]{m}{\fdiv[k]{m - (n-1)}{\vec i}} = i + n^kR \geq \vec i$.
\end{proof}

We say that a index function $P$ \emph{decouples} if there exist functions 
$P_j:\Z\to\Z$ such that $[P(\vec i)]_\ell=P_\ell(i_\ell)$, for $\ell\in[1:d]$. 
As well, we say that $P$ is \emph{non decreasing} if $P(\vec i)\leq P(\vec j)$ 
for $\vec i \leq \vec j$.

\begin{lemma}\label{lem:decouples}
	Let $F$ be a box function given by $F(\vec i)=(P(\vec i):Q(\vec i))$ such that 
$P$ and $Q$ are non decreasing and decouple. If $Q_\ell(i_\ell) - 
P_\ell(i_\ell+1)\geq -1$, for $j=[1:d]$. Then $F$ is box preserving, i.e. 
$F([\vec i:\vec j])=(P(\vec i):R(\vec j))$.
\end{lemma}
\begin{proof}
	Let $\vec z \in F[\vec i:\vec j]$, then there exists $\vec w \in [\vec i:\vec 
j]$ such that $\vec z \in (P(\vec w):Q(\vec w))$, i.e. $P(\vec w)\leq \vec z 
\leq Q(\vec w)$. Since $P$ and $Q$ are non decreasing $P(\vec i)\leq P(\vec 
w)\leq \vec z \leq Q(\vec w) \leq Q(\vec j)$. Hence $\vec z \in (P(\vec 
i):Q(\vec j))$, therefore $F[\vec i:\vec j]\subset (P(\vec i): Q(\vec j))$.
	
	On the other hand, let $\vec z \in (P(\vec i):Q(\vec j))$ then $P(\vec i)\leq 
\vec z\leq Q(\vec j)$. Since $P$ and $Q$ decouple, for each $j\in[1:d]$ we have 
that $P_\ell(i_\ell)\leq z_\ell \leq Q_\ell(j_\ell)$. Besides $Q_\ell(i_\ell) - 
P_\ell(i_\ell+1)\geq -1$, for $\ell=[1:d]$, then there exists $w_\ell \in 
[i_\ell:j_\ell]$ such that $P_\ell(w_\ell)\leq z_\ell\leq Q_\ell(w_\ell)$. 
Hence, $\vec w=(w_1,\ldots,w_d)\in [\vec i: \vec j]$ satisfies that $P(\vec 
w)\leq \vec z\leq Q(\vec w)$. Therefore, $\vec z \in F[\vec i:\vec j]$.
\end{proof}

\begin{lemma}\label{lem:box_preserving_composition}
	If $F$ and $G$ are box preserving functions, then the composition  $F\circ G$ 
is a box function and evenmore is box preserving.
\end{lemma}
\begin{proof}
	Let $P_F, Q_F, P_G$ and $Q_G$ index functions such that $F(\vec i)=(P_F(\vec 
i), Q_F(\vec i))$ and $G(\vec i)=(P_G(\vec i), Q_G(\vec i))$. Notice that 
$F\circ G(\vec i) = F(G(\vec i))=F(P_G(\vec i): Q_G(\vec i))$, since $F$ is box 
preserving then $F\circ G(\vec i) = (P_F\circ P_G(\vec i):Q_F\circ Q_G(\vec 
i))$. Hence, $F\circ G$ is a box function. 
	Moreover, $F\circ G[\vec i:\vec j]= F(G[\vec i:\vec j])$, since $F$ and $G$ 
are box preserving $F\circ G[\vec i:\vec j]= F(P_G(\vec i):Q_G(\vec 
j))=(P_F\circ P_G(\vec i):Q_F\circ Q_G(\vec j))$, i.e. $F\circ G$ is box 
preserving.
\end{proof}
\begin{corollary}\label{lem:box_preserving_iteration}
	If $F(\vec i)=(P(\vec i):Q(\vec i))$ is box preserving, then the $k$-th 
iterate of $F$ satisfies that $F^k(\vec i)=(P^k(\vec i), Q^k(\vec i))$ and is 
box preserving.
\end{corollary}

%--------------------------------------------------------------------------
\bibliographystyle{siam}
\bibliography{biblio}
%--------------------------------------------------------------------------
\end{document}